\documentclass[11pt]{amsart}
\usepackage{latexsym,amsfonts,amssymb,amsthm,amsmath,mathrsfs,color,amscd,graphicx,fullpage,parskip,hyperref,bm,bbm,dsfont}

\usepackage[pagewise]{lineno}%\linenumbers

\usepackage{comment}
\includecomment{uncomment}
\usepackage{commath,mathtools,setspace}
\usepackage{paralist}
\usepackage{extarrows}
\usepackage{tikz}

\excludecomment{comment} % for comments we want to exclude

\newcommand{\s}[1]{{\mathcal #1}}
\newcommand{\sr}[1]{{\mathscr #1}}
\newcommand{\bb}[1]{{\mathbb #1}}

\newcommand{\ceil}[1]{\left\lceil #1 \right\rceil}

\newtheorem{theorem}{Theorem} %[section]
\newtheorem{corollary}[theorem]{Corollary}

\newtheorem{lemma}[theorem]{Lemma}
\newtheorem{proposition}[theorem]{Proposition}

\newtheorem{definition}[theorem]{Definition}

\newtheorem{remark}[theorem]{Remark}
\newtheorem{assumption}[theorem]{Assumption}

\numberwithin{equation}{section}
\numberwithin{theorem}{section}

\newcounter{step}
\begin{document}
	
	\title%[]
	{A ``trembling hand perfect'' equilibrium for a certain class of mean field games}
	
	\author{P.~Jameson Graber}
	\thanks{The author is grateful to be supported by National Science Foundation through NSF Grant DMS-2045027.}
	\address{J. Graber: Baylor University, Department of Mathematics;\\
		Sid Richardson Building\\
		1410 S.~4th Street\\
		Waco, TX 76706\\
		Tel.: +1-254-710- \\
		Fax: +1-254-710-3569 
	}
	\email{Jameson\_Graber@baylor.edu}
	
	%\dedicatory{In honor of Professor Alain Bensoussan, on the occasion of his 85th birthday}
	
	%\author{}
	
	 %35Q89 = mfg PDE, 49N80 = mfg calc var opt ctrl, 91A16 = mfg game theory
	\date{\today}   
	
	%\begin{comment}
	\begin{abstract}
		We study a particular class of mean field games whose solutions can be formally connected to a scalar transport equation on the Wasserstein space of measures.
		For this class, we construct some interesting explicit examples of non-uniqueness of Nash equilibria.
		We then address the selection problem of finding rational criteria by which to choose one equilibrium over others.
		We show that when the theory of entropy solutions is used, we can obtain explicit error estimates for the ``vanishing noise limit,'' where the error is measured in a certain norm that measures the distance between two functions over the set of empirical measures.
	\end{abstract}
	
	%\end{comment}

	\maketitle
	
	%\tableofcontents
	
	\section{Introduction} \label{sec:introduction}
	
	Mean field game theory was introduced by Lasry and Lions \cite{lasry07} and independently by Caines Huang and Malham\'e \cite{huang2006large,huang2007large} as a way to model the contiuum limit of $N$-player Nash games as $N \to \infty$.
	An overview of the theory can be found in the following monographs, surveys, and lecture notes: \cite{carmona2017probabilistic,carmona2017probabilisticII,bensoussan2013mean,gomes2014mean,cardaliaguet2010notes,cardaliaguet2021cetraro}.
	Briefly, a mean field Nash equilibrium is understood conceptually as a population distribution in which no individual has any incentive (as measured by a cost function) to change behavior.
	In Section \ref{sec:setting} below we will give a precise mathematical description of a mean field game, at least in one particular framework.
	
	Most of the literature on mean field games has been occupied with the problem of proving there exists a unique solution, i.e.~a unique mean field Nash equilibrium, and then characterizing it as the limit of either open-loop or closed-loop equilibria for $N$-player games (see e.g.~\cite{cardaliaguet2019master}), or else as an $\varepsilon$-Nash equilibrium (i.e.~an ``almost equilibrium") to the $N$-player game.
	However, there is an increasing number of results showing non-uniqueness of the equilibrium for various types of mean field games \cite{bardi2019non,bayraktar2020non,cirant2016stationary,cirant2019time,cecchin2019convergence}, which has led to an interest in ``selection principles" \cite{cecchin2022selection,delarue2020selection}.
	
	The concept of a selection principle is essential to the justification of Nash equilibrium as a concept, since without uniqueness, each player is left without a rational way of selecting a strategy \cite{harsanyi1988general}.
	The idea introduced by Selten \cite{selten1975}, which is now often called ``trembling hand perfect equilibrium,'' is to refine the concept of equilibrium to include only those strategy distributions that are robust in the face of a ``perturbed'' game in which all players use mixed strategies and employ every strategy with at least probability $\varepsilon$.
	Drawing inspiration from this idea, we will offer in this study a possible selection principle based on the same principle of ``vanishing noise.''
	See \cite{graber2025remarks} for further discussion in a broader context.
	It is important to stress that, although the purpose of a selection principle is to uniquely determine an equilibrium, selection principles are themselves not unique.
	Each selection principle can be viewed as a modeling choice, which may have certain advantages and disadvantages in various contexts.
	
	The idea of selecting the vanishing noise limit as a weak solution is common in the mathematical theory of PDE.
	In the context of Hamilton-Jacobi equations, for example, it gives rise to the viscosity solutions \cite{crandall1983viscosity}, whereas for nonlinear balance laws one can arrive at entropy solutions \cite{kruvzkov1970first,kruvzkov1967generalized}.
	The concept of entropy solutions has recently been used by Cecchin and Delarue to define weak solutions to master equations for potential mean field games \cite{cecchin2022weak}.
	Indeed, the connection between mean field games and entropy solutions to classical nonlinear transport equations seems to be a common theme in recent studies of selection principles \cite{cecchin2019convergence,cecchin2022weak,delarue2020selection,graber2024some}.
	This theme is a source of inspiration for the main results of this paper.
	
	In this study, we will examine only a specific class of mean field games, which is taken from a recent article by the author and Meszaros \cite{graber2024some}.
	This class is conveniently simple in one sense, since the Nash equilibrium is determined entirely by a single real number.
	When taken as a function of the initial distribution of states and the time horizon, this real number turns out to satisfy a nonlinear transport equation on the space of probability measures, which can be viewed as a simplified (or even ``factorized'' \cite{graber2024some}) master equation.
	The study of uniqueness of mean field equilibria is equivalent to studying whether or not characteristics of the transport equation cross each other.
	
	In Section \ref{sec:uniqueness and non} we will construct some explicit examples where uniqueness of the equilibrium fails, both for $N$-player games and for the mean field game.
	Then in Section \ref{sec:trembling hand} we will offer a selection principle based on a particular stochastic perturbation of the mean field game, letting the noise vanish.
	Our principle will be based on the theory of entropy solutions for a sequence of discretized versions of our master equation.
	The main result is an error estimate in a certain norm, which measures the distance between two functions over the set of empirical measures; the statement is found in Section \ref{sec:PDE interpretation}, and its proof is given in Section \ref{sec:estimates}.
	Finally, in Section \ref{sec:conclusion}, we provide some concluding remarks and open questions.

	\subsection{The setting} \label{sec:setting}
	Consider the following class of mean field games.
	For $t > 0$ and $x \in \bb{R}^d$, we will define a value function
	\begin{equation} \label{eq:value function}
		u(t,x) = \inf \cbr{\int_0^t L(x(s),\dot{x}(s)) \dif s + g\del{\sigma_0(\mu_0),x(0)} : x(\cdot) \in \s{C}^1\del{[0,t];\bb{R}^d}, \ x(t) = x}
	\end{equation}
	for a given flow of measures $(\mu_s)_{s \in [0,t]}$ where $\mu_t = m$ is given.
	Here the cost depends on $(\mu_s)$ only through the final distribution $\mu_0$, via the function
	\begin{equation}
		\sigma_0 : \sr{P}_2(\bb{R}^d) \to \bb{R}.
	\end{equation}
	Here $\sr{P}_2(\bb{R}^d)$ is the set of all Borel probability measures on $\bb{R}^d$ with finite second moments.
	
	Nash equilibrium is defined as follows.
	Starting from a distribution $m$, all players choose optimal curves that cause an evolution of the population (backwards in time) so that at time $0$ the players are distributed according to some new measure $m'$.
	If $m' = \mu_0$, the game is said to be in equilibrium. 
	Formally, this is equivalent to finding a solution to the usual PDE system
	\begin{equation} \label{eq:mfg}
		\begin{aligned}
			\partial_t u + H(x,Du) &= 0,\\
			\partial_t \mu + \operatorname{div}\del{D_p H(x,Du)\mu} &= 0,\\
			u(0,x) &= g\del{\sigma_0(\mu_0),x},\\
			\mu_t &= m.
		\end{aligned}
	\end{equation}
	The coupling occurs only at time $0$, which is simpler than most mean field games.
	There is a further simplification in that the equilibrium is entirely characterized by the scalar quantity $\sigma = \sigma_0(\mu_0)$.
	
	Throughout this manuscript we will assume that the optimal control problem in \eqref{eq:value function} is always uniquely solvable.
	In particular, let us define 
	\begin{equation} \label{eq:parametrized value function}
		v(t,\sigma,x) = \inf \cbr{\int_0^t L(x(s),\dot{x}(s)) \dif s + g\del{\sigma,x(0)} : x(\cdot) \in \s{C}^1\del{[0,t];\bb{R}^d}, \ x(t) = x}.
	\end{equation}
	Then we assume $v$ is a \emph{classical} solution to the Hamilton-Jacobi equation
	\begin{equation}
		\partial_t v + H(x,Dv) = 0, \quad v(0,\sigma,x) = g(\sigma,x).
	\end{equation}
	There is a unique optimal trajectory that solves \eqref{eq:parametrized value function} given by solving
	\begin{equation} \label{optimal trajectory}
		\dot{x}(s) = D_p H(x(s),Dv(s,\sigma,x(s))), \quad x(t) = x.
	\end{equation}
	We denote the solution by $x(s,t,\sigma,x)$.
	Then we can see that the mean field Nash equilibrium condition is given by
	\begin{equation} \label{eq:equilibrium}
		\sigma = \sigma_0\del{x(0,t,\sigma,\cdot)_\sharp m}.
	\end{equation}
	Here $f_\sharp m$ denotes the push-forward of $m$ through a Borel measurable function $f$, i.e.~$f_\sharp m(E) = m(f^{-1}(E))$.
	If $\sigma$ is a solution of \eqref{eq:equilibrium}, then \eqref{eq:mfg} is solved by setting $u(t,x) = v(t,\sigma,x)$ and $\mu_s = x(s,t,\sigma,\cdot)_\sharp m$.
	Conversely, if $(u,\mu)$ is a solution of \eqref{eq:mfg}, then \eqref{eq:equilibrium} is solved by setting $\sigma = \sigma_0(\mu_0)$.
	
	Mean field games are supposed to be the limits of $N$-player games as $N \to \infty$.
	Let us now define the $N$-player game.
	Let $\bm{x} = (x_1,\ldots,x_N) \in (\bb{R}^d)^N$ be given.
	Denote by $\bm{x}_{-i}$ the vector $(x_1,\ldots,x_{i-1},x_{i+1},\ldots,x_N)$, and denote by $(\bm{x}_{-i},y_i)$ the vector $(x_1,\ldots,x_{i-1},y_i,x_{i+1},\ldots,x_N)$.
	For $i = 1,\ldots,N$ define a cost
	\begin{equation}
		J_i(\bm{x}(\cdot)) = \int_0^t L(x_i(s),\dot{x}_i(s))\dif s + g\del{\sigma_0\del{\frac{1}{N-1}\sum_{j \neq i} \delta_{x(0,t,\sigma_j,x_j)}},x_i(0)}
	\end{equation}
	over all $\s{C}^1$ curves $\bm{x}(\cdot):[0,t] \to (\bb{R}^d)^N$ such that $\bm{x}(t) = \bm{x}$.
	The curve $\widehat{\bm{x}}(\cdot)$ is a Nash equilibrium (in open loop controls) provided
	\begin{equation}
		J_i\del{\widehat{\bm{x}}(\cdot)} = \inf_{x_i(\cdot)} J_i\del{\widehat{\bm{x}}_{-i}(\cdot),x_i(\cdot)} \quad \forall i = 1,\ldots,N.
	\end{equation}
	The $N$-player Nash system is given by
	\begin{equation} \label{eq:N player equilibrium}
		\sigma_i = \sigma_0\del{\frac{1}{N-1}\sum_{j \neq i} \delta_{x(0,t,\sigma_j,x_j)}}, \quad i = 1,\ldots,N.
	\end{equation}
	If system \eqref{eq:N player equilibrium} has a solution, then a Nash equilibrium $\bm{x}(\cdot)$ is given by setting $x_i(s) = x(s,t,\sigma_i,x_i)$ for each $i$.
	Conversely, if $\bm{x}(\cdot)$ is a Nash equilibrium, then $\sigma_i = \sigma_0\del{\frac{1}{N-1}\sum_{j \neq i} \delta_{x_j(0)}}$ is a solution of \eqref{eq:N player equilibrium}.
	
	The purpose of this study is to examine the uniqueness of solutions of both the $N$-player game and the mean field game.
	If there is more than one equilibrium, then we are interested in the following question: 
	is there a selection principle that rationally distinguishes between equilibria, and if so, do we get convergence of selected equilibria to the mean field limit?

	\subsection{Notations and basic definitions} \label{sec:notation}
	
	The Wasserstein space \(\mathscr{P}_2(\mathbb{R}^d)\) is defined as the set of Borel probability measures on \(\mathbb{R}^d\) with finite second moment, i.e.,
	\[
	\mathscr{P}_2(\mathbb{R}^d) = \left\{ \mu \in \mathscr{P}(\mathbb{R}^d) \;\middle|\; \int_{\mathbb{R}^d} \abs{x}^2 \, \dif\mu(x) < \infty \right\}.
	\]
	The Wasserstein distance of order 2 between two probability measures \(\mu, \nu \in \mathscr{P}_2(\mathbb{R}^d)\) is given by
	\[
	W_2(\mu, \nu) = \left( \inf_{\gamma \in \Pi(\mu, \nu)} \int_{\mathbb{R}^d \times \mathbb{R}^d} \abs{x - y}^2 \, \dif\gamma(x,y) \right)^{\frac{1}{2}},
	\]
	where \(\Pi(\mu, \nu)\) is the set of all couplings of \(\mu\) and \(\nu\), i.e., the set of probability measures on \(\mathbb{R}^d \times \mathbb{R}^d\) with marginals \(\mu\) and \(\nu\).
	The integral on the right-hand side is called the (quadratic) transportation cost.
	The set of optimal couplings, denoted $\Pi_0(\mu,\nu)$, is the set of all $\pi \in \Pi(\mu,\nu)$ that minimize the transportation cost; this set is always nonempty \cite{villani2008optimal}.
	
	Let $u:\sr{P}_2(\bb{R}^d) \to \bb{R}$.
	We say that $u$ is \emph{continuously differentiable} provided there exists a continuous function $D_m u:\sr{P}_2(\bb{R}^d) \times \bb{R}^d \to \bb{R}^d$ such that
	\begin{equation}
		u(\nu) - u(\mu) = \int_{\bb{R}^d} D_m u(\mu,x)\cdot (y-x)\dif \pi(x,y) + o\del{W_2(\mu,\nu)}
	\end{equation}
	for some (hence for all) $\pi \in \Pi_0(\mu,\nu)$.
	See e.g.~\cite{carmona2017probabilistic,gangbo2019differentiability} for a discussion of Wasserstein derivatives and conditions for smoothness.
	A simple example is $u(\mu) = \int \phi \dif \mu$ where $\phi:\bb{R}^d \to \bb{R}$ is a smooth function, in which case $D_m u(\mu,x) = D\phi(x)$.
	
	A vector in $\bb{R}^d$ will be denoted $x = (x^1,\ldots,x^d)$.
	A vector in $(\bb{R}^d)^N$ will be denoted by $\bm{x} = (x_1,\ldots,x_n)$, where each $x_j \in \bb{R}^d$ can be written $x_j = (x_j^1,\ldots,x_j^d)$.
	For $\bm{x} \in (\bb{R}^d)^N$, the $n$ point empirical measure corresponding to $\bm{x}$ is $\mu^N_{\bm{x}} = \frac{1}{N}\sum_{j = 1}^N \delta_{x_j}$, where $\delta_x$ is a Dirac mass at the point $x \in \bb{R}^d$.

	For a continuous function $u:\sr{P}_2(\bb{R}^d) \to \bb{R}$, the discretization to $n$ point empirical measures is defined by
	\begin{equation}
		u_N(\bm{x}) = u(\mu^N_{\bm{x}}).
	\end{equation}
	If $u$ is continuously differentiable, then so is $u_N$, with
	\begin{equation} \label{discrete derivative}
		D_{x_j} u_N(\bm{x}) = \frac{1}{N}D_m u(\mu^N_{\bm{x}},x_j).
	\end{equation}
	See, e.g.~\cite[Proposition 5.39]{carmona2017probabilistic}.

	\section{Uniqueness and non-uniqueness}
	
	\label{sec:uniqueness and non}
	
	\subsection{Conditions leading to uniqueness of the mean field equilibrium}
	
	\label{sec:uniqueness conditions}
	
	As discussed in \cite{graber2024some}, the equilibrium condition \eqref{eq:equilibrium} has at most one solution under a suitable assumption of monotonicity.
	In particular, if the function
	\begin{equation} \label{sigma - sigma0}
		F(\sigma) := \sigma - \sigma_0\del{x(0,t,\sigma,\cdot)_\sharp m}
	\end{equation}
	is strictly monotone, there is at most one Nash equilibrium.
	To give a sufficient condition that guarantees uniqueness, we start by setting
	\begin{equation}
		b(t, \sigma, x) = D_{p} H(x, Dv(t, \sigma, x)).
	\end{equation}
	\begin{assumption}
		\label{sigma0 bounded}		
		$\sigma_0:\sr{P}_2(\bb{R}^d) \to \bb{R}$ is Borel measurable and bounded, taking values in some interval $(\sigma_*,\sigma^*)$.
	\end{assumption}
	\begin{assumption}[Regularity of $\sigma_0$]
		\label{sigma0 smooth}
		$\sigma_0:\sr{P}_2(\bb{R}^d) \to \bb{R}$ is continuously differentiable, and $D_m \sigma_0$ is bounded.
		Moreover, for each $N$, $\sigma_{0,N}(\bm{x}) = \sigma_0(\mu^N_{\bm{x}})$ is twice continuously differentiable with H\"older continuous second derivatives.
	\end{assumption}
	We do not always insist on Assumption \ref{sigma0 smooth}, especially in discussing uniqueness and non-uniqueness, since regularity is not as essential as monotonicity.
	Some particularly interesting examples of non-uniqueness arise when $\sigma_0$ is discontinuous; see Section \ref{sec:nonuniqueness}.
	
	\begin{assumption}[Regularity of $b$]
		\label{b regular}		
		$b:\intco{0,\infty} \times (\sigma_*,\sigma^*) \times \bb{R}^d \to  \bb{R}^d$, written $b(t,\sigma,x)$, is a continuous vector field, continuously differentiable with respect to variables $(\sigma,x)$ such that $b$ is uniformly bounded and Lipschitz with respect to $x$, i.e.~there exist constants $M_b$ and $L_b$ such that
		\begin{equation} \label{b lipschitz constant}
			\abs{b(t,\sigma,x)} \leq M_b, \ \ \abs{D_x b(t,\sigma,x)} \leq L_b \quad \forall x \in \bb{R}^d, \ \forall t \geq 0, \ \forall \sigma \in (\sigma_*,\sigma^*).
		\end{equation}
		We also assume $b$ is H\"older continuous with respect to $t$ on bounded intervals $[0,T]$, uniformly with respect to the variables $(\sigma,x)$.
		Finally, let $B(t,\sigma,x) = \int_0^\sigma b(t,u,x)\dif u$.
		Then there exists a constant $H_B$ such that
		\begin{equation} \label{B Hessian bound}
			\abs{D_{xx}^2 B(t,\sigma,x)} \leq H_B \quad \forall x \in \bb{R}^d, \ \forall t \geq 0, \ \forall \sigma \in (\sigma_*,\sigma^*).
		\end{equation}
	\end{assumption}
	
	\begin{remark}
		It follows from \eqref{b lipschitz constant} that $D_x B$ is also bounded.
		A sufficient condition implying \eqref{B Hessian bound} is a bound on $D_{xx}^2 b$.
	\end{remark}
	
	Unlike Assumption \ref{sigma0 smooth}, Assumption \ref{b regular} is rather crucial in this study.
	It guarantees that, given a parameter $\sigma$ and time horizon $t$, the optimal trajectory $x(s,t,\sigma,x)$ is uniquely defined and continuous in all of its arguments.
	
	\begin{definition} \label{def:pos vector}
		Let $v = (v^1,\ldots,v^d) \in \bb{R}^d$.
		We will write $v \geq 0$ and say that $v$ is non-Negative provided $v^j \geq 0$ for all $j = 1,\ldots,d$.
		If $v^j > 0$ for all $j$ then we write $v > 0$ and say it is positive.
		We say $v$ is non-positive, resp.~negative, if $-v$ is non-Negative, resp.~positive.
	\end{definition}
	
	\begin{assumption}[Monotonicity of $\sigma_0$]
		\label{monotonicity}
		$\sigma_0$ is ``increasing in all directions,'' i.e.
		\begin{equation}
			D_m \sigma_0(m,x) \geq 0 \quad \forall x \in \bb{R}^d, \quad \forall m \in \sr{P}_2(\bb{R}^d),
		\end{equation}
		cf.~Definition \ref{def:pos vector}.
	\end{assumption}
	
	\begin{assumption}[Monotonicity of $b$]
		\label{monotonicity b}
		Each component of $b$ is strictly increasing with respect to $\sigma$, i.e.~$\partial_\sigma b > 0$ in the sense of Definition \ref{def:pos vector}.
	\end{assumption}
	
	\begin{assumption}
		\label{diagonal}
		The vector field $b(t,x,u)$ is \emph{diagonal} in the sense that the $j$th component function can be written $b^j(t,x^j,u)$ where $x^j$ is the $j$th component of $x$.
		In particular, $D_x b(t,x,u)$ is a diagonal matrix.
	\end{assumption}
	Assumption \ref{monotonicity b} can be thought of as a monotonicity assumption on $b$. However, it only guarantees uniqueness when accompanied by Assumption \ref{monotonicity}.
	
	A simple example that satisfies Assumptions \ref{sigma0 bounded}, \ref{sigma0 smooth}, \ref{b regular}, \ref{monotonicity}, \ref{monotonicity b}, and \ref{diagonal} is $d = 1$, $H(x,p) = \frac{1}{2}p^2$ and $g(\sigma,x) = \sigma x$ so that $b(t,\sigma,x) = \sigma t$, and $\sigma_0(m) = \int \phi \dif m$ where $\phi$ is a smooth, bounded, increasing function.
	More generally, we have the following:
	\begin{proposition}
		Let $H(x,p) = \frac{1}{2}\abs{p}^2$.
		Suppose $g(\sigma,x)$ satisfies the following:
		\begin{enumerate}
			\item $D_{xx}^2 g(\sigma,x)$ is diagonal, i.e.~$D_{x^jx^k}^2 g(\sigma,x) = 0$ for $j \neq k$;
			\item $D_{x^jx^j}^2 g(\sigma,x) \geq 0$ for all $j = 1,\ldots,d$ (hence $g$ is convex with respect to $x$);
			\item $D_{xx}^2 g(\sigma,x)$ and $D_{xxx}^3 g(\sigma,x)$ are bounded;
			\item $D_{x^j \sigma}^2 g(\sigma,x) > 0$ for all $j = 1,\ldots,d$.
		\end{enumerate}
		Then $b$ satisfies all the conditions specified in Assumptions \ref{b regular}, \ref{monotonicity b}, and \ref{diagonal}.
		
		Let $\sigma_0(m) = \int_{(\bb{R}^d)^N} \phi(x_1,\ldots,x_n)\dif m(x_1) \cdots \dif m(x_n)$ where $\phi$ is a smooth, bounded function that is increasing in all variables, i.e.~$D_{x_j} \phi \geq 0$ for all $j = 1,\ldots,N$ (see Definition \ref{def:pos vector}).
		Then $\sigma_0$ satisfies the conditions specified by Assumptions \ref{sigma0 bounded}, \ref{sigma0 smooth}, and \ref{monotonicity}.
	\end{proposition}
	
	\begin{proof}
		In this case we have the Hopf-Lax formula
		\begin{equation}
			v(t,\sigma,x) = \min\cbr{\frac{\abs{x-y}^2}{2t} + g(\sigma,y) : y \in \bb{R}^d}.
		\end{equation}
		One computes that the optimal velocity is given by
		\begin{equation}
			b(t,\sigma,x) = D_x v(t,\sigma,x) = \frac{x - \del{I + tD_x g(\sigma,\cdot)}^{-1}(x)}{t} =: \frac{x - y(t,\sigma,x)}{t}.
		\end{equation}
		By implicit differentiation and the conditions given on $g$, one sees that $y$ is smooth and has derivatives given by
		\begin{equation}
			\begin{split}
				\partial_{x^j}y^k &= \frac{\delta_{jk}}{1 + tD_{y^ky^k}^2 g(\sigma,y)},\\
				\partial_{\sigma}y^k &= \frac{-tD_{y^k\sigma}^2 g(\sigma,y)}{1 + tD_{y^ky^k}^2 g(\sigma,y)}.
			\end{split}
		\end{equation}
		From this we deduce that $b = D_y g(\sigma,y)$ satisfies the conditions specified in Assumptions \ref{b regular}, \ref{monotonicity b}, and \ref{diagonal}.
		
		As for $\sigma_0$, we note that $\sigma_0$ is smooth with Wasserstein gradient given by
		\begin{multline}
			D_m \sigma_0(m,x)
			\\ = \sum_{j=1}^N \int_{(\bb{R}^d)^{N-1}} D_{x_j}\phi(x_1,\ldots,x_{j-1},x,x_{j+1},\ldots,x_n)\dif m(x_1) \cdots \dif m(x_{j-1}) \dif m(x_{j+1}) \cdots \dif m(x_n).
		\end{multline}
		Since $\phi$ is increasing in all variables, so is $\sigma_0$.
		Therefore Assumptions \ref{sigma0 bounded}, \ref{sigma0 smooth}, and \ref{monotonicity} are satisfied.
	\end{proof}
	
	The assumption that $b$ should be increasing with respect to $\sigma$ can be interpreted as saying that optimal velocity is an increasing function of the coordination parameter $\sigma$.
	Assumption \ref{diagonal} is introduced for technical reasons, as we will see presently.
	
	Under these assumptions, it is possible to show that $\partial_{\sigma} x(0,t,\sigma,x) < 0$ (in the sense that every component is negative).
	Indeed, by implicitly differentiating the ODE \eqref{optimal trajectory}, one computes 
	\begin{equation} \label{partial sigma x}
		\partial_{\sigma} x(s,t,\sigma,x) = -\int_s^t \exp\del{-\int_s^\tau D_x b(\tau',x_j(\tau',t,x,\sigma),\sigma)\dif \tau'}\partial_\sigma b(\tau,x(\tau,t,x,\sigma),\sigma)\dif \tau,
	\end{equation}
	and for $s < t$ this is negative by Assumption \ref{monotonicity}.
	Therefore, with $F$ defined as in \eqref{sigma - sigma0}, we see that
	\begin{equation} \label{Fprime}
		F'(\sigma) = 1 - \int_{\bb{R}^d} D_m \sigma_0\del{x(0,t,\sigma,\cdot)_\sharp m, x(0,t,\sigma,x)}\partial_\sigma x(0,t,\sigma,x)\dif m(x) \geq 1,
	\end{equation}
	hence $F$ necessarily has precisely one zero in the range of $\sigma_0$.
	Therefore Assumptions \ref{sigma0 bounded}, \ref{sigma0 smooth}, \ref{b regular}, \ref{monotonicity}, \ref{monotonicity b}, and \ref{diagonal} give us unconditional uniqueness for the equilibrium, in the sense that it holds no matter the size of $t$.
	
	Throughout the rest of this manuscript, we will always take Assumption \ref{b regular}, concerning the regularity of the optimal vector field $b$.
	Assumptions \ref{sigma0 bounded} and \ref{sigma0 smooth}, concerning boundedness and regularity of $\sigma_0$, are usually in force when we seek to prove general results; however, they can often be relaxed in the context of finding explicit Nash equilibria.
	On the other hand, Assumption \ref{monotonicity}, on the monotonicity of $\sigma_0$, as well as Assumptions \ref{monotonicity b} and \ref{diagonal} on the structure of $b$, are only in force when proving uniqueness of a mean field equilibrium.
	Otherwise, these assumptions may be dispensed with.
	
	Let us remark that the monotonicity conditions presented here have no necessary connection to other more commonly used conditions, such as the Lasry-Lions monotonicity condition or displacement monotonicity.
	That is, Assumptions \ref{monotonicity} and \ref{monotonicity b} do not imply that the coupling function $g(\sigma_0(m),x)$ must be monotone in a Lasry-Lions or displacement sense, and vice versa. 
	This is made clear by the examples found in \cite{graber2024some,graber2023monotonicity}.
	Conceptually, the monotonicity conditions presented here are comparable to the classical conditions leading to global in time smoothness to solutions of nonlinear conservation laws; see e.g.~\cite[Theorem 6.1.1]{dafermos16}.
	
	In the remainder of this section, we examine examples of non-uniqueness, which is our motivation for studying selection principles.
	In deriving examples, we also notice that the $N$-player game can have equilibria that do not converge to mean field equilibria as $N \to \infty$.
	This phenomenon has been noticed elsewhere; see for instance \cite{lacker2016general,cardaliaguet2020multiple}, and see \cite{campi2018nplayer,nutz2020convergence,cecchin2019convergence,lacker2020convergence} for related results.

	\subsection{Uniqueness for the mean field does not imply uniqueness for $N$ players}
	
	\label{sec:mean field unique but not N player}
	
	Since Assumption \ref{monotonicity} guarantees uniqueness for the mean field game, one could reasonably ask whether it guarantees uniqueness for the corresponding $N$-player game.
	Unfortunately, the answer is no.
	In this section we give a simple example to prove this.
	
	Consider the case $d = 1$, $H(p) = \frac{1}{2}p^2$, $g(\sigma,x) = x\sigma$, and $\sigma_0(m) = \int \phi \dif m$ for a function $\phi$ that is at least \emph{piecewise} smooth.
	We will assume $\phi' > 0$, so that the monotonicity from Assumption \ref{monotonicity} holds; in particular, the mean field equilibrium is unique.
	Indeed, the mean field equilibrium condition is now precisely
	\begin{equation}
		\sigma = \int_{\bb{R}} \phi(x-\sigma t)\dif m(x),
	\end{equation}
	and one can check directly that there is a unique solution by virtue of the fact that $\phi$ is non-decreasing.
	To simplify even further, we will assume $m$ is a Dirac mass at zero, so the mean field equilibrium condition is nothing more than
	\begin{equation}
		\sigma = \phi(-\sigma t),
	\end{equation}
	which clearly has a unique solution.
	On the other hand, for the $N$-player game, condition \eqref{eq:N player equilibrium} becomes
	\begin{equation}
		\sigma_i = \frac{1}{N-1}\sum_{k \neq i}\phi(-\sigma_k t),
	\end{equation}
	where for simplicity we just assume that all players start at $x_i = 0$, which is the most natural choice given that they should converge on the Dirac mass at zero in the mean field limit.
	If $\sigma$ is a mean field equilibrium, then one can take $\sigma_i = \sigma$ for all $i$ to get an $N$-player equilibrium.
	However, it is in general not the only equilibrium.
	
	Let us suppose that $\phi$ is an odd function, i.e.~$\phi(x) = -\phi(-x)$, such that $\phi(x)/x$ is unbounded on $(0,\infty)$.
	(For example, one may take $\phi(x) = x^\alpha$ for $x > 0$ where $\alpha \in (0,1) \cup (1,\infty)$.)
	Then there exists a sequence $b_N > 0$ such that $\phi(tb_N) = (N-1)b_N$ for all $N$.
	Suppose that $N$ is even.
	For $k$ an odd number, we take $\sigma_k = b_N$, whereas for $k$ even we take $\sigma_k = -b_N$.
	Since $\phi$ is an odd function, we have that
	\begin{equation}
		\sum_{k = 1}^N \phi(-t\sigma_k) = 0,
	\end{equation}
	the even terms canceling with the odd terms.
	It follows that
	\begin{equation}
		\frac{1}{N-1}\sum_{k \neq i} \phi(-t\sigma_k) = -\frac{1}{N-1}\phi(-t\sigma_i) = \frac{1}{N-1}\phi(t\sigma_i).
	\end{equation}
	Since $\phi(tb_N) = (N-1)b_N$ and $\phi$ is odd, we deduce that $\frac{1}{N-1}\phi(t\sigma_i) = \sigma_i$ for all $i$, hence
	\begin{equation}
		\frac{1}{N-1}\sum_{k \neq i} \phi(-t\sigma_k) = \sigma_i,
	\end{equation}
	as desired.	
	
	This example can be used to show that $N$-player equilibria need not converge to mean field equilibria.
	If $\phi(x)/x$ is unbounded as $x \to \infty$, this means $b_N$ can be chosen such that $b_N \to \infty$.
	Thus the final measure is $\frac{1}{2}\delta_{-tb_N} + \frac{1}{2}\delta_{tb_N}$, which diverges.
	The issue here is a lack of compactness: the mass escapes to infinity.
	Alternatively, if $\phi(x) \to \infty$ as $x \to M$ for some $M \in (0,\infty)$, we can choose $b_N \to M$ and get a final measure that converges to $\frac{1}{2}\delta_{-tM} + \frac{1}{2}\delta_{tM}$, which is not a mean field equilibrium.
	The issue in that case is that $\phi$ is singular.
	Note that both of these examples violate the boundedness Assumption \ref{sigma0 bounded}.
	Below we will give examples that need not violate this assumption, though they will assume singular data.

	\subsection{Non-uniqueness for both kinds of equilibrium}
	
	\label{sec:nonuniqueness}
	
	In this section we give examples for which both mean field and $N$-player games can have multiple equilibria.
	First, let us show at least one generic condition under which it is possible to derive an $N$-player equilibrium from a mean field equilibrium.
	\begin{proposition} \label{pr:N player from mfg}
		Suppose $\sigma$ is a solution of \eqref{eq:equilibrium}.
		Let $m$ be a probability measure with compact support and let $\bm{x}^N = (x_1^N,\ldots,x_N^N) \in (\bb{R}^d)^N$ be a sequence such that each $x_j^N$ is in the support of $m$ and $W_2\del{\frac{1}{N}\sum_{j=1}^N \delta_{x_j^N},m} \to 0$.
		Define $m_0 = x(0,t,\sigma,\cdot)_\sharp m$.
		Assume the function $\sigma_0$ is constant in a neighborhood of $m_0$.
		Then there exists $N_0$ sufficiently large so that if $N \geq N_0$, then $\sigma_i = \sigma$ for all $i$ gives a solution to \eqref{eq:N player equilibrium}.
	\end{proposition}
	
	\begin{proof}
		Let $y_i^N = x(0,t,\sigma,x_i^N)$.
		Since $x(0,t,\sigma,\cdot)$ is Lipschitz, it follows that $W_2\del{\frac{1}{N}\sum_{j=1}^N \delta_{y_j^N},m_0} \to 0$, and also that $m_0$ has compact support.
		Notice that each $y_j^N$ is in the support of $m_0$.
		It follows that
		\begin{equation}
			W_2\del{\frac{1}{N}\sum_{j=1}^N \delta_{y_j^N},\frac{1}{N-1}\sum_{j\neq i} \delta_{y_j^N}}^2
			\leq \frac{1}{N(N-1)}\sum_{j \neq i} \abs{y_j - y_i}^2
			\leq \frac{D^2}{N},
		\end{equation}
		where $D$ is the diameter of the support of $m_0$.
		Let $U$ be a neighborhood of $m_0$ such that $\sigma_0$ is constant on $U$.
		It follows that there exists $N_0$ such that if $N \geq N_0$, we have $\frac{1}{N-1}\sum_{j\neq i} \delta_{y_j^N} \in U$ for all $i = 1,\ldots,N$, which implies
		\begin{equation}
			\sigma = \sigma_0\del{\frac{1}{N-1}\sum_{j\neq i} \delta_{y_j^N}}
			= \sigma_0\del{\frac{1}{N-1}\sum_{j\neq i} \delta_{x(0,t,\sigma,x_j^N)}}.
		\end{equation}
		Thus $\sigma_i = \sigma$ solves system \eqref{eq:N player equilibrium}.
	\end{proof}
	
	It is straightforward to find examples of non-uniqueness.
	Let us consider the simple case where $d = 1$, $H(x,p) = \frac{1}{2}p^2$, and $g(\sigma,x) = \sigma x$.
	Then optimal trajectories are lines with $x(0,t,\sigma,x) = x - \sigma t$.
	The mean field Nash equilibrium condition becomes
	\begin{equation}
		\sigma = \sigma_0((\tau_{-\sigma t})_\sharp m),
	\end{equation}
	where $\tau_y(x) := x+y$.
	We will furthermore specify that $\sigma_0(m)$ depends only on the barycenter of $m$, i.e.
	\begin{equation} \label{sigma0=f}
		\sigma_0(m) = f\del{\int_{\bb{R}} x \dif m(x)}.
	\end{equation}
	Hence the Nash equilibrium condition becomes
	\begin{equation} \label{simple equilibrium}
		\sigma = f\del{\int_{\bb{R}} x \dif m(x) - \sigma t}.
	\end{equation}
	As explained in \cite{graber2024some}, solutions of \eqref{simple equilibrium} are closely related to solutions of the Burgers equation
	\begin{equation}
		\partial_t \sigma + \sigma \partial_\xi\sigma = 0, \quad \sigma(0,\xi) = f(\xi),
	\end{equation}
	where the variable $\xi$ stands in for $\int_{\bb{R}} x \dif m(x)$.
	We deduce the existence of multiple equilibria in the well-known cases where the Burgers equation has crossing characteristics, namely whenever $f$ is not a monotone increasing function.
	To give a concrete example, let us consider
	\begin{equation} \label{f=1,0}
		f(\xi) = \begin{cases}
			1, &\text{if}~\xi < 0,\\
			0, &\text{if}~\xi \geq 0.
		\end{cases}
	\end{equation}
	If $0 \leq \int_{\bb{R}} x \dif m(x) < t$, then $\sigma = 0$ and $\sigma = 1$ are both solutions of \eqref{simple equilibrium}.
	
	Moreover, if $0 < \int_{\bb{R}} x \dif m(x) < t$, then by Proposition \ref{pr:N player from mfg} we see that the corresponding $N$-player game can be solved with either $\sigma_i = 0$ for all $i$ or $\sigma_i = 1$ for all $i$, for $N$ sufficiently large.
	There are thus two distinct equilibria for the $N$-player game, each of which converges to one of the two mean field equilibria.
	However, there may be other equilibria for the $N$-player game, which do not converge to either of the mean field equilibria.
	We now present one such example.
	
	\begin{proposition} \label{prop:3 equil}
		Assume that $x(0,t,\sigma,x) = x - \sigma t$ and $\sigma_0$ is given by \eqref{sigma0=f} and \eqref{f=1,0}.
		Fix $t > 0$.
		Let $a,b > 0$ be given.
		For $N \in \bb{N}$, $j = 1,\ldots,N$ define
		\begin{equation}
			\begin{split}
				J_N &:= \ceil{\frac{(N-1)b}{a+b}},\\
				x_j^N &:= \begin{cases}
					a+t &\text{if}~j \leq J_N,\\
					-b &\text{if}~j > J_N
				\end{cases}
			\end{split}
		\end{equation}
		where $\ceil{x}$ means the smallest integer greater than or equal to $x$.
		Then for $N$ large, there are at least three solutions to the $N$-player Nash equilibrium condition \eqref{eq:N player equilibrium}, which are given by
		\begin{enumerate}
			\item $\sigma_i = 0$ for $i = 1,\ldots,N$;
			\item $\sigma_i = 1$ for $i = 1,\ldots,N$; and
			\item $\sigma_i = 1$ for $i = 1,\ldots,J_N$ and $\sigma_i = 0$ for $i = J_N + 1,\ldots,N$.
		\end{enumerate}
	\end{proposition}
	\begin{proof}
		Set $\lambda = \frac{b}{a+b}$.
		Let $m = \lambda \delta_{a+t} + (1-\lambda)\delta_{-b}$, where each $\delta_x$ is a Dirac mass at $x$.
		Then we have
		\begin{equation}
			\frac{1}{N}\sum_{j=1}^N \delta_{x_j^N} = \frac{J_N}{N}\delta_{a+t} + \del{1 - \frac{J_N}{N}}\delta_{-b} \to m
		\end{equation}
		as $N \to \infty$.
		We also have
		\begin{equation}
			\int_{\bb{R}} x \dif m(x) = \lambda(a+t) - (1-\lambda)b = t\lambda \in (0,t),
		\end{equation}
		so Proposition \ref{pr:N player from mfg} implies that there are at least two equilibria for $N$ large: $\sigma_i = 0$ for all $i$ or $\sigma_i = 1$ for all $i$ (this can easily be checked directly).
		
		Now set $\sigma_i = 1$ for $i = 1,\ldots,J_N$ and $\sigma_i = 0$ for $i = J_N + 1,\ldots,N$.
		We will now show this is also an equilibrium.
		First rewrite \eqref{eq:N player equilibrium} as
		\begin{equation} \label{simple N player}
			\sigma_i = f\del{\frac{1}{N-1}\sum_{j \neq i} (x_j - \sigma_j t)} = \begin{cases}
				1, &\text{if}~\sum_{j \neq i} (x_j - \sigma_j t) < 0,\\
				0, &\text{if}~\sum_{j \neq i} (x_j - \sigma_j t) \geq 0
			\end{cases}
		\end{equation}
		using \eqref{sigma0=f} and \eqref{f=1,0}.
		To see that \eqref{simple N player} holds true, note that $\sum_{j = 1}^N (x_j - \sigma_j t) = \sum_{j=1}^{J_N} (a+t-t) + \sum_{j=J_N+1}^{N} (-b) = J_N a - (N-J_N)b$.
		We deduce that
		\begin{equation} \label{eq:sumxj-sjt}
			\begin{split}
				\sum_{j \neq i} (x_j - \sigma_j t) &= \begin{cases}
					(J_N-1) a - (N - J_N)b, &\text{if}~i \leq J_N,\\
					J_N a - (N - J_N - 1)b,
					&\text{if}~i > J_N
				\end{cases}\\
				&= \begin{cases}
					(a+b)(J_N-1) - (N-1)b, &\text{if}~i \leq J_N,\\
					(a+b)J_N - (N-1)b,
					&\text{if}~i > J_N.
				\end{cases}
			\end{split}
		\end{equation}
		Since $J_N = \ceil{\frac{(N-1)b}{a+b}}$, we see that $\sum_{j \neq i} (x_j - \sigma_j t) < 0$ for $i \leq J_N$ and $\sum_{j \neq i} (x_j - \sigma_j t) \geq 0$ for $i > J_N$, so \eqref{simple N player} holds.
		Thus we have an equilibrium.
	\end{proof}
	
	The first two equilibria given by Proposition \ref{prop:3 equil} naturally give final measures that converge to mean field equilibrium measures.
	However, the third equilibrium does not.
	In this case the final measure is given by
	\begin{equation} \label{anamolous limit}
		\frac{1}{N}\sum_{j = 1}^N \delta_{x_j - \sigma_j t} = \frac{J_N}{N}\delta_a + \del{1-\frac{J_N}{N}}\delta_{-b}
		\overset{N \to \infty}{\longrightarrow} \lambda \delta_a + (1-\lambda)\delta_{-b}.
	\end{equation}
	Yet the two mean field equilibria correspond to the measures
	\begin{equation}
		\lambda \delta_{a+t} + (1-\lambda)\delta_{-b} \quad (\text{for} \ \sigma = 0) \quad \text{and} \quad \lambda \delta_{a} + (1-\lambda)\delta_{-b-t} \quad (\text{for} \ \sigma = 1). 
	\end{equation}
	Thus the limit of the third equilibrium is a sort of mixture of two mean field equilibria.
	The interpretation is that the players at point $a+t$ expect the parameter $\sigma$ to be $1$, hence they should move to the left by $\sigma t = t$, whereas the players at point $-b$ expect $\sigma = 0$, hence they should not move.
	
	Notice that the barycenter of the anomalous limit $\lambda \delta_a + (1-\lambda)\delta_{-b}$ is precisely zero, which is the point of discontinuity of the function $f$.
	The existence of the first two equilibria in Proposition \ref{prop:3 equil} is not a consequence of the discontinuity, but merely of the fact that $f$ is decreasing; this is the same phenomenon as when characteristics of the Burgers equation cross.
	It is this third, anomalous equilibrium that emerges as a result of the discontinuity.
	When $f$ is continuous, the results of Lacker \cite{lacker2016general} prove that equilibria for $N$-player games must converge to (weak) mean field equilibria.
	
	Let us generalize Proposition \ref{prop:3 equil} to measures whose support consists of more than two points, and to more general optimal trajectories $x(s,t,\sigma,x)$.
	
	\begin{proposition} \label{prop:general 3 equil}
		Let $d=1$, and let $b(t,\sigma,x)$ satisfy Assumptions \ref{b regular} and \ref{monotonicity b} (Assumption \ref{diagonal} holds trivially as $d = 1$).
		
		Fix $t > 0$.
		Let $\tilde \mu$ and $\tilde \nu$ be probability measures of compact support such that
		\begin{equation} \label{supports}
			\operatorname{supp} \tilde\mu \subset \intoc{-\infty,0}, \quad \operatorname{supp} \tilde\nu \subset \intco{c,\infty}
		\end{equation}
		for some $c > 0$, and such that $\int x \dif \tilde \mu(x) < 0$.
		Assume also that $\sigma_0$ is given by \eqref{sigma0=f} and \eqref{f=1,0}.
		Let $\lambda \in (0,1)$ be chosen such that 
		\begin{equation} \label{Normalization lambda}
			\int_{\bb{R}} x \dif\del{(1-\lambda) \tilde\mu + \lambda \tilde \nu} = 0.
		\end{equation}
		Set $\mu = x(t,0,0,\cdot)_\sharp \tilde \mu$, $\nu = x(t,0,1,\cdot)_\sharp \tilde \nu$, and $m = (1-\lambda) \mu + \lambda \nu$.
		Then there exists an infinite sequence $N_k$ such that for $N = N_k$, there exists a discrete measure $\mu^{N_k}_{\bm{x}_k}$ such that $W_2(\mu^{N_k}_{\bm{x}_k},m) \to 0$ and such that
		there are at least three solutions to the $N$-player Nash equilibrium condition \eqref{eq:N player equilibrium}, which are given by
		\begin{enumerate}
			\item $\sigma_i = 0$ for $i = 1,\ldots,N$;
			\item $\sigma_i = 1$ for $i = 1,\ldots,N$; and
			\item $\sigma_i = 1$ for $i = 1,\ldots,J$ and $\sigma_i = 0$ for $i = J + 1,\ldots,N$
		\end{enumerate}
		where $1 < J < N$ satisfies $J \approx \lambda N$.
	\end{proposition}

	\begin{proof}
		%[Proof of Proposition \ref{prop:general 3 equil}]
		Define
		\begin{equation} \label{difference of tilde nu mu}
			\delta := \int x \dif (\tilde \nu - \tilde \mu)(x)
			\geq c.
		\end{equation}
		For a given, $N$ we will choose $J$ such that
		\begin{equation} \label{J condition}
			0 \leq J - \lambda N < \frac{c}{\delta}.
		\end{equation}
		In other words, we choose $J = \ceil{\lambda N}$, and we must have $N$ such that $\ceil{\lambda N} - \lambda N < \frac{c}{\delta}$.
		Note that there are infinitely many $N$ for which this holds.
		
		Let $\frac{1}{J}\sum_{j=1}^J \delta_{y_j}$ be the optimal discretization with $J$ points of $\tilde \nu$, given in Section \ref{sec:optimal discretization}; likewise, let $\frac{1}{N-J}\sum_{j=J+1}^N \delta_{y_j}$ be the optimal discretization with $N-J$ points of $\tilde \mu$.
		For $j = 1,\ldots,J$ we define $x_j = x(t,0,1,y_j)$ so that $y_j = x(0,t,1,x_j)$; likewise, for $j = J+1,\ldots,N$ we define $x_j = x(t,0,0,y_j)$ so that $y_j = x(0,t,0,x_j)$.
		
		Our first objective is to show that $W_2\del{\frac{1}{N}\sum_{j=1}^N \delta_{x_j},m} \to 0$ as $N \to \infty$.
		We have that
		\begin{equation}
			\begin{split}
				W_2\del{\frac{1}{J}\sum_{j=1}^J \delta_{x_j},\nu}
				&\leq \operatorname{Lip}(x(t,0,1,\cdot))W_2\del{\frac{1}{J}\sum_{j=1}^J \delta_{y_j},\tilde\nu},\\
				W_2\del{\frac{1}{N-J}\sum_{j=J+1}^N \delta_{x_j},\mu}
				&\leq \operatorname{Lip}(x(t,0,0,\cdot))W_2\del{\frac{1}{N-J}\sum_{j=J+1}^N \delta_{x_j},\tilde \mu},
			\end{split}
		\end{equation}
		and by the results of Section \ref{sec:optimal discretization} these go to zero as $N \to \infty$.
		
		Now we claim that $\mu$ and $\nu$ have disjoint supports.
		To prove this, note that Assumption \ref{b regular} allows us to show that $x(s,t,\sigma,x)$ is increasing in the $x$-variable; indeed, it is straightforward to prove that
		\begin{equation}
			\partial_x x(s,t,\sigma,x) = \exp\del{\int_t^s \partial_x b(\tau,\sigma,x(\tau,t,\sigma,x))\dif \tau} > 0.
		\end{equation}
		It follows from \eqref{supports} that
		\begin{equation}
			\operatorname{supp} \mu \subset \intoc{-\infty,x(t,0,0,0)}, \quad
			\operatorname{supp} \nu \subset \intco{x(t,0,1,c),\infty}.
		\end{equation}
		By \eqref{partial sigma x} with Assumptions \ref{b regular} and \ref{monotonicity b} it follows that $\partial_\sigma x(t,0,\sigma,x) > 0$, and so in particular
		\begin{equation}
			x(t,0,1,c) > x(t,0,0,c) > x(t,0,0,0).
		\end{equation}
		Thus $\operatorname{supp} \mu \cap \operatorname{supp} \nu = \emptyset$.
		
		Now if $T_\nu : \bb{R} \to \bb{R}$ is a transport map such that $(T_\nu)_\sharp \nu = \frac{1}{J}\sum_{j=1}^J \delta_{x_j}$ and $T_\mu : \bb{R} \to \bb{R}$ is a transport map such that $(T_\mu)_\sharp \nu = \frac{1}{N-J}\sum_{j=1}^J \delta_{x_j}$, then we can construct a transport map $T:\bb{R} \to \bb{R}$ such that $T = T_\nu$ on $\operatorname{supp} \nu$ and $T = T_\mu$ on $\operatorname{supp} \mu$.
		Then $T_\sharp m = \frac{\lambda}{J}\sum_{j=1}^J \delta_{x_j} + \frac{1-\lambda}{N-J}\sum_{j=1}^J \delta_{x_j}$ and
		\begin{equation}
			\int \abs{T(x) - x}^2 \dif m(x) = \lambda \int \abs{T_\nu(x) - x}^2 \dif \nu(x)
			+ (1-\lambda)\int \abs{T_\mu(x) - x}^2 \dif \mu(x).
		\end{equation}
		Since $T_\mu$ and $T_\nu$ are arbitrary, it follows that
		\begin{equation}
			W_2\del{\frac{\lambda}{J}\sum_{j=1}^J \delta_{x_j} + \frac{1-\lambda}{N-J}\sum_{j=1}^J \delta_{x_j},m}^2
			\leq \lambda W_2\del{\frac{1}{J}\sum_{j=1}^J \delta_{x_j},\nu} + (1-\lambda)W_2\del{\frac{1}{N-J}\sum_{j=1}^J \delta_{x_j},\mu}^2,
		\end{equation}
		which becomes arbitrarily small as $N \to \infty$.
		Finally, we construct a transport plan from $\frac{\lambda}{J}\sum_{j=1}^J \delta_{x_j} + \frac{1-\lambda}{N-J}\sum_{j=1}^J \delta_{x_j}$ to $\frac{1}{N}\sum_{j=1}^N \delta_{x_j}$ as follows.
		Noting that $\frac{\lambda}{J} \leq \frac{1}{N}$ and $\frac{1-\lambda}{N-J} \geq \frac{1}{N}$, we let $\pi$ be a discrete measure such that 
		\begin{equation}
			\begin{split}
				\pi(\{x_j,x_j\}) &= \frac{\lambda}{J}, \quad j = 1,\ldots,J,\\
				\pi(\{x_j,x_j\}) &= \frac{1}{N}, \quad j = J+1,\ldots,N,\\
				\pi(\{x_j,x_k\}) &= \frac{J - \lambda N}{NJ(N-J)}, \quad j = J+1,\ldots, N, \ k = 1,\ldots,J.
			\end{split}
		\end{equation}
		This is a probability measure that couples $\frac{\lambda}{J}\sum_{j=1}^J \delta_{x_j} + \frac{1-\lambda}{N-J}\sum_{j=1}^J \delta_{x_j}$ to $\frac{1}{N}\sum_{j=1}^N \delta_{x_j}$, and so
		\begin{equation} \label{discrete measures estimate}
			\begin{split}
				W_2\del{\frac{\lambda}{J}\sum_{j=1}^J \delta_{x_j} + \frac{1-\lambda}{N-J}\sum_{j=1}^J \delta_{x_j},\frac{1}{N}\sum_{j=1}^N \delta_{x_j}}^2
				&\leq \int \abs{x-y}^2 \dif \pi(x,y)\\
				&= \sum_{j=J+1}^N \sum_{k=1}^J \abs{x_j - x_k}^2 \frac{J-\lambda N}{NJ(N-J)}\\
				&\leq \sum_{j=J+1}^N \sum_{k=1}^J 2(x_j^2 + x_k)^2 \frac{J-\lambda N}{NJ(N-J)}\\
				&= \frac{2(J-\lambda N)}{N(N-J)}\sum_{j=J+1}^N x_j^2 
				+ \frac{2(J-\lambda N)}{NJ}\sum_{k=1}^J x_k^2.
			\end{split}
		\end{equation}
		Since $J - \lambda N \leq 1$, $\frac{1}{N-J}\sum_{j=J+1}^N x_j^2 \to \int x^2 \dif \mu(x)$ and $\frac{1}{J}\sum_{j=1}^J x_j^2 \to \int x^2 \dif \nu(x)$ as $N \to \infty$, it follows that the right-hand side of \eqref{discrete measures estimate} goes to zero as $N \to \infty$.
		We conclude that $W_2\del{\frac{1}{N}\sum_{j=1}^N \delta_{x_j},m} \to 0$ as $N \to \infty$.
		
		Now we prove that the three equilibria defined in the Proposition are indeed solutions to \eqref{eq:N player equilibrium}.
		Equivalently, we need to show
		\begin{equation} \label{general 3 equil cond}
			\sigma_i = f\del{\frac{1}{N-1}\sum_{j \neq i} x(0,t,\sigma_j,x_j)} = \begin{cases}
				1, &\text{if}~\sum_{j \neq i}x(0,t,\sigma_j,x_j) < 0,\\
				0, &\text{if}~\sum_{j \neq i}x(0,t,\sigma_j,x_j) \geq 0.
			\end{cases}
		\end{equation}
		
		To see that the constant $\sigma_i = 0$ is an equilibrium, recall that, by \eqref{partial sigma x}, $x(0,t,\sigma,x)$ is strictly decreasing with respect to $\sigma$, hence
		\begin{equation}
			\begin{split}
				\int x \dif \del{x(0,t,0,\cdot) \sharp m}(x)
				&= \int x(0,t,0,x) \dif m(x)\\
				&= \lambda \int x(0,t,0,x) \dif \nu(x)
				+ (1-\lambda)\int x(0,t,0,x) \dif \mu(x)\\
				&> \lambda \int x(0,t,1,x) \dif \nu(x)
				+ (1-\lambda)\int x(0,t,0,x) \dif \mu(x)\\
				&= \lambda \int x \dif \tilde \nu(x) + (1-\lambda) \int x \dif \tilde \mu(x) = 0.
			\end{split}
		\end{equation}
		Hence $\sigma_0(\tilde m) = f(\int x \dif \tilde m(x)) = 0$ for every $\tilde m$ in a neighborhood of $x(0,t,0,\cdot)_\sharp m$.
		We now appeal to Proposition \ref{pr:N player from mfg} to see that $\sigma_i = 0$ is an equilibrium for $N$ large enough.
		An analogous argument shows that the constant $\sigma_i = 1$ is also an equilibrium for $N$ large enough.
		
		Finally, we let $\sigma_i = 1$ for $i = 1,\ldots,J$ and $\sigma_i = 0$ for $i = J+1,\ldots,N$.
		To show that \eqref{general 3 equil cond} holds, we need to show that
		\begin{equation}
			\sum_{j \neq i} x(0,t,\sigma_j,x_j) \begin{cases}
				< 0, &\text{when}~i \leq J,\\
				\geq 0, &\text{when}~i > J.
			\end{cases}
		\end{equation}
		Equivalently, we need to show
		\begin{equation}
			\sum_{j = 1}^J x(0,t,1,x_j)
			+ \sum_{j = J+1}^N x(0,t,0,x_j)
			\begin{cases}
				< x(0,t,1,x_i), &\text{when}~i \leq J,\\
				\geq x(0,t,0,x_i), &\text{when}~i > J,
			\end{cases}
		\end{equation}
		which, by recalling the definitions of $x_j$ in terms of $y_j$, is simply
		\begin{equation} \label{critical discrete inequality}
			\sum_{j = 1}^J y_j
			+ \sum_{j = J+1}^N y_j
			\begin{cases}
				< y_i, &\text{when}~i \leq J,\\
				\geq y_i, &\text{when}~i > J.
			\end{cases}
		\end{equation}
		We recall that $\frac{1}{J}\sum_{j=1}^J \delta_{y_j}$ is the optimal discretization of $\tilde \nu$ as given in Section \ref{sec:optimal discretization}, and by the construction there we can see that $y_i \geq c$ for all $i \leq J$ because $\operatorname{supp} \tilde \nu \subset \intco{c,\infty}$.
		Similarly, $\frac{1}{N-J}\sum_{j=J+1}^N \delta_{y_j}$ is the optimal discretization of $\tilde \mu$, and it follows that $y_i \leq 0$ for $i > J$.
		Hence a sufficient condition for \eqref{critical discrete inequality} is
		\begin{equation} \label{critical discrete sufficient}
			0 \leq \sum_{j = 1}^J y_j
			+ \sum_{j = J+1}^N y_j < c
		\end{equation}
		We also have, cf.~\eqref{same barycenter},
		\begin{equation}
			\frac{1}{J}\sum_{j = 1}^J y_j = \int x \dif \tilde \nu(x),
			\quad \frac{1}{N-J}\sum_{j=J+1}^N y_j = \int x \dif \tilde \mu(x).
		\end{equation}
		Hence \eqref{critical discrete sufficient} becomes
		\begin{equation}
			0 \leq J \int x \dif \tilde\nu(x) + (N-J)\int x \dif \tilde \mu(x) < c.
		\end{equation}
		Subtracting $N\lambda \int x \dif \tilde \nu(x) + N(1-\lambda)\int x \dif \tilde \mu(x) = 0$ from this inequality, we see that it is equivalent to \eqref{J condition}, where $\delta$ is given by \eqref{difference of tilde nu mu}.
		The proof is complete.
	\end{proof}

	\begin{remark}
		\begin{enumerate}
			\item There is nothing special about the values 1 and 0 for $\sigma$; replace the function $f$ with a function that jumps from $\alpha$ to $\beta$ with $\alpha > \beta$ and the result will remain essentially the same.
			\item Similarly, there is nothing special about where the jump occurs.
			If $f$ jumps at a point $\xi$ instead of at 0, then replace the intervals $\intoc{-\infty,0}$ and $\intco{c,\infty}$ in \eqref{supports} with $\intoc{-\infty,\xi}$ and $\intco{\xi + c,\infty}$, respectively; and choose $\lambda$ such that the right-hand side of \eqref{Normalization lambda} is $\xi$ instead of $\lambda$.
			\item The condition \eqref{supports} essentially defines a \emph{necessary} condition for the construction of a third equilibrium that is a mixture of the other two.
			This becomes clear by observing first that \eqref{critical discrete inequality} is a necessary condition, and in order to cover all possible cases it implies
			\begin{equation}
				\max \operatorname{supp} \tilde \mu \leq J \int x \dif \tilde\nu(x) + (N-J)\int x \dif \tilde \mu(x) < \min \operatorname{supp} \tilde \nu.
			\end{equation}
			Given the postulate that $J \approx \lambda N$, dividing by $N$ and letting $N \to \infty$, we derive condition \eqref{Normalization lambda}.
			Then, as a corollary, it becomes necessary that 
			\begin{equation}
				\max \operatorname{supp} \tilde \mu \leq 0 < \min \operatorname{supp} \tilde \nu,
			\end{equation}
			which implies \eqref{supports}.
			\item In light of this last point, the interpretation of Proposition \ref{prop:general 3 equil} is as follows: we have a population which is in fact split into two, and these are separated by a sufficiently large distance so as to confuse each side.
			That is, each side anticipates the wrong strategy from the other side, and this is made possible precisely because of the distance between them.
			\item As a result, the limit of this third equilibrium will not in fact be a mean field equilibrium, since it gives a measure that cannot be written as the push-forward of the initial measure through a map $x(0,t,\sigma,\cdot)$.
		\end{enumerate}
	\end{remark}
	
	\subsection{A selection principle for the $N$-player game} \label{sec:N player selection}
	In light of the discussion of this section, we will propose a modest ``selection principle" for the $N$-player game, which is that \emph{we should at least eliminate those equilibria that do not converge to a mean field equilibrium}.
	Generically, we expect all $N$-player equilibria to converge to a mean field equilibrium \cite{lacker2016general}, which we can see here by a direct argument.
	Suppose $\sigma_0$ is continuous and bounded, and let $m$ be a probability measure with compact support, with empirical measure $\frac{1}{N}\sum_{j=1}^N\delta_{x_j^N}$ converging to $m$.
	Suppose $\sigma_1^N,\ldots,\sigma_N^N$ is an $N$-player equilibrium, i.e.~a solution to \eqref{eq:N player equilibrium}.
	Notice that
	\begin{equation}
		\sigma_i^N = \sigma_0\del{x(0,t,\cdot,\cdot)_\sharp \frac{1}{N-1}\sum_{j \neq i}\delta_{\sigma^N_j,x_j^N}}.
	\end{equation}
	If we pass to a subsequence on which $\frac{1}{N}\sum_{j=1}^N\delta_{\sigma_j^N,x_j^N}$ has a cluster point, which we call $\mu$, then on that subsequence we get, following the proof of Proposition \ref{pr:N player from mfg}, 
	\begin{equation}
		\sigma_i^N \to \sigma_0\del{x(0,t,\cdot,\cdot)_\sharp \mu},
	\end{equation}
	uniformly in $i$.
	But this implies that $\mu = \delta_\sigma \times m$ where $\sigma$ satisfies
	\begin{equation}
		\sigma = \sigma_0\del{x(0,t,\sigma,\cdot)_\sharp m},
	\end{equation}
	i.e.~the $N$-player equilibrium converges to a mean field equilibrium.
	
	The examples where this fails are examples where $\sigma_0$ fails to satisfy either boundedness or continuity.
	The way in which they fail can be interpreted as introducing an asymmetry: two segments of the population are distinguished and choose complementary strategies to form an equilibrium.
	Therefore, our selection principle for $N$-player equilibrium is really an insistence on complete symmetry, such that the forming of ``factions'' is not allowed.

	\section{The ``trembling hand''}
	
	\label{sec:trembling hand}
	
	The purpose of this section is to investigate one method of selecting from among multiple mean field equilibria.
	The inspiration for our approach is the notion introduced by Selten \cite{harsanyi1988general,selten1975}, which is now commonly called ``trembling hand perfect equilibrium."
	To understand this concept, we first need to define an $\varepsilon$-perturbation of a game.
	For a finite game (i.e.~finite number of strategies), this can be defined by changing the game so that every player must use mixed strategies, assigning probabilities of order $\varepsilon$ to every possible strategy.
	A trembling hand perfect equilibrium can be defined as a strategy distribution that continues to be an equilibrium for all $\varepsilon$-perturbed games with $\varepsilon$ small enough.
	This definition is not suitable for differential games, in which the space of strategies is infinite and even infinite-dimensional.
	We propose instead to define $\varepsilon$-perturbations of the game as follows.
	As we have done throughout this article, we assume that each player always has a unique best response, in the form of a feedback control, to each possible population behavior.
	To perturb the game, we posit that each player's trajectory will be driven by this feedback control plus $\varepsilon$ times a Brownian motion.
	In other words, each player intends to play rationally in a deterministic setting, but their actual behavior is stochastically perturbed.
	We will admit a certain mean field Nash equilibrium as a ``trembling hand perfect" equilibrium if it is robust with respect to this perturbation.
	
	Let us make this notion mathematically precise.
	Recall that \( b(t, \sigma, x) = D_{p} H(x, Dv(t, \sigma, x)) \) is the optimal response vector field.
	Assumptions \ref{sigma0 bounded}, \ref{sigma0 smooth}, \ref{b regular}, \ref{monotonicity b}, and \ref{diagonal} are in force.
	Let us introduce an additional assumption, which will make it easier to prove the main result of this section:
	\begin{assumption} \label{sigma0 cpct supp}
		There exists $R > 0$ such that
		\begin{equation}
			\operatorname{supp} \sigma_0 \subset \cbr{m \in \sr{P}_2(\bb{R}^d) : \int_{\bb{R}^d} \abs{x}^2 \dif m(x) \leq R^2}.
		\end{equation}
	\end{assumption}
	If we start with a $\sigma_0$ that satisfies Assumptions \ref{sigma0 bounded} and \ref{sigma0 smooth}, we can modify it to satisfy \ref{sigma0 cpct supp} by using a cut-off function.
	To wit, let $\phi:\bb{R} \to \bb{R}$ be a smooth function supported in $[-R^2,R^2]$, and define $\tilde \sigma_0(m) = \phi\del{\int \abs{x}^2 \dif m(x)}\sigma_0(m)$.
	Then $\tilde{\sigma}_0$ satisfies all three Assumptions \ref{sigma0 bounded}, \ref{sigma0 smooth}, and \ref{sigma0 cpct supp}.
	
	The mean field equilibrium condition \eqref{eq:equilibrium} can be expressed as the following system of equations:
	\begin{equation} \label{characteristics}
		\begin{cases} 
			\dot{x}(s) = b(s, \sigma, x(s)), & x(t) = x \\
			\sigma = \sigma_{0}(x(0, t, \sigma, \cdot)_\sharp m). &
		\end{cases}
	\end{equation}
	In order to introduce noise, we face a difficulty: the space of probability measures is infinite-dimensional.
	Thus it is not clear how to define a noise that acts on all players both independently and all at once. (See, however, some work in this direction by Delarue \cite{delarue2019restoring}.)
	Our solution will be to discretize, considering only distributions that are empirical measures:
	\begin{equation} \label{characteristics discrete}
		\begin{cases} 
			\dot{x}_i(s) = b(s, \sigma, x_i(s)), & x_i(t) = x_i \\
			\sigma = \sigma_{0,N}\del{\bm{x}(0)} &
		\end{cases}
	\end{equation}
	where we recall that $\sigma_{0,N}(\bm{x}) := \frac{1}{N}\sum_{j=1}^N \delta_{x_j}$.
	One could choose to think of \eqref{characteristics discrete} as an $N$-player game in which player $i$ reacts to the measure $\frac{1}{N}\sum_{j=1}^N \delta_{x_j}$, which includes herself, instead of $\frac{1}{N-1}\sum_{j \neq i} \delta_{x_j}$ as in \eqref{eq:N player equilibrium}.
	This forces all players to calculate the same parameter $\sigma$, thus following the selection principle from Section \ref{sec:N player selection}.
	However, for our purposes, \eqref{characteristics discrete} is nothing more than a special case of the mean field equilibrium, in which the population distribution is an empirical measure characterized by the vector $\bm{x} = (x_1,\ldots,x_N)$.
	
	To add noise, let $(W(s))_{s \geq -t}$ be a standard Brownian motion in $(\bb{R}^d)^N$ adapted to a filtration $(\s{F}_s)_{s \geq -t}$.
	We set
	\begin{equation} \label{time reversal}
		x_{i}(s) = \bar{x}_{i}(-s), \quad \sigma(s) = \bar{\sigma}(-s)
	\end{equation}
	where $\bar{\bm{x}}, \bar \sigma$ solve the system
	\begin{equation} \label{trembling hand}
		\begin{aligned}
			\dif \bar x_{i}(s) &= -b(-s, \bar \sigma(s), \bar x_{i}(s)) \, \dif s + \varepsilon^N \, \dif W_i(s), \quad \bar x_{i}(-t) = x_{i},\\
			\bar \sigma(s) &= E  \sbr{\sigma_{0,N}\del{\bm{x}(0)} \middle| \s{F}_s} , \quad -t \leq s \leq 0.
		\end{aligned}
	\end{equation}
	
	The interpretation is as follows.
	Instead of trajectories $x_i(s)$ following precisely the optimal dynamics for a given $\sigma$, their trajectory is perturbed by a Brownian motion. Then, instead of $\sigma$ being a constant determined by \eqref{characteristics discrete}, it is a martingale, which is obtained at each time by conditioning $\sigma_{0, N}(\bm{x}(0))$ on the knowledge player $i$ has up to the present time. We have reversed time with the substitutions \eqref{time reversal}; this was necessary only because the PDEs studied in Section \ref{sec:PDE interpretation} have forward-in-time solutions. The parameter $\varepsilon > 0$ is a small parameter, such that $\varepsilon^N$ measures the strength of the noise. We note that as $N \to \infty$, this strength decreases exponentially. We will see later precisely what role $\varepsilon$ plays in the selection of equilibria.

	The noise introduced here is not the same thing as an $N$-player game with idiosyncratic noise. In an $N$-player game with noise, each player selects a strategy that optimizes the expected value of the objective functional over all stochastic curves. In this setting, by contrast, the noise represents a stochastic deviation from what would otherwise be a deterministic optimal velocity. It represents the idea that, although the parameter $\sigma$ might be given and the corresponding optimal trajectory $x_i(s)$ definitely computable, player $i$ will not adhere to the optimal trajectory with absolute certainty. This is the ``trembling hand.''
	(It is also not to be confused with common noise, since, although it affects the whole population distribution at once, it does so in a way that affects each Dirac mass independently.)
	
	For each $t$ and for each $\bm{x} \in (\bb{R}^d)^N$, let $\sigma = \sigma(t,\bm{x})$ be an equilibrium solving \eqref{characteristics discrete} and let $\sigma_\varepsilon = \sigma_\varepsilon(t,\bm{x})$ be a solution of \eqref{trembling hand}.
	For the purposes of this study, we will say that $\sigma$ is a trembling hand perfect equilibrium if the limit $\lim_{\varepsilon \to 0} \sigma_\varepsilon = \sigma$ holds in some suitable sense (see Theorem \ref{thm:error estimates}).
	However, it needs to be emphasized that such a definition is far from unique; there are in fact many ways of selecting from among multiple mean field Nash equilibria.
	In particular, we have chosen to perturb the path of the empirical measure by making it a martingale. If we instead allow other sorts of stochastic perturbations, then we can get any number of selection principles, as in \cite{lions2023linear}.
	See \cite{graber2025remarks} for a discussion.

	\subsection{PDE interpretation and main result} \label{sec:PDE interpretation}

	In this study, the selection of Nash equilibria will be directly linked to the theory of entropy solutions for certain nonlinear transport equations.
	To wit, we see that System \eqref{characteristics} comprises the characteristics of the transport equation
	\begin{equation} \label{transport in W_2}
		\begin{aligned}
			\partial_{t} \sigma(t, m) + \int D_{m} \sigma(t, m, x) b(t, \sigma(t, m, x), x) \, \dif m(x) &= 0,\\
			\sigma(0, m) &= \sigma_0(m).
		\end{aligned}
	\end{equation}
	See \cite{graber2024some} for more details.
	In particular, we derive the following corollary from \cite[Theorem 3.11]{graber2024some}:
	\begin{lemma}
		Let Assumptions \ref{sigma0 bounded}, \ref{sigma0 smooth}, \ref{monotonicity}, \ref{b regular}, \ref{monotonicity b}, and \ref{diagonal} all be in force.
		Then there exists a unique classical solution of \eqref{transport in W_2}.
	\end{lemma}
	
	\begin{proof}
		Recall from Section \ref{sec:uniqueness conditions} that the function $F$ defined in \eqref{sigma - sigma0} satisfied $F' \geq 1$.
		This implies \cite[Assumption 3.4]{graber2024some}.
		The remaining assumptions in that paper follow directly from the assumptions given here.
		Therefore, by \cite[Theorem 3.11]{graber2024some}, there exists a unique classical solution of \eqref{transport in W_2}.
	\end{proof}

	Equation \eqref{transport in W_2} is a transport PDE on $(0,\infty) \times \sr{P}_2(\bb{R}^d)$, where $\sr{P}_2(\bb{R}^d)$ is the Wasserstein space of probability measures.\

	Recall the discretization rule Equation \eqref{discrete derivative} from Section \ref{sec:notation}.
	Accordingly, let us define
	\begin{equation}
		\sigma_{N}(t, \bm{x}) := \sigma\del{t, \frac{1}{N}\sum_{j=1}^{N} \delta_{x_{j}}} \quad \forall \bm{x} = (x_1,\ldots,x_N) \in (\bb{R}^d)^N.
	\end{equation}
	
	If $\sigma$ is a classical solution of \eqref{transport in W_2}, then $\sigma_N$ satisfies the discretized version
	\begin{equation} \label{transport in R^dN}
		\begin{aligned}
			\partial_{t} \sigma_{N}(t, \bm{x}) + \sum_{j=1}^{N} D_{x_j} \sigma_{N}(t, \bm{x}) b(t, \sigma_{N}(t, \bm{x}), x_{j}) &= 0,\\
			\sigma_{N}(0, \bm{x}) &= \sigma_{0, N}(\bm{x}),
		\end{aligned}
	\end{equation}
	which is a classical PDE on $(0,\infty) \times (\bb{R}^d)^N$.

	As for the ``trembling hand," it is known that if we take arbitrary $(t,\bm{x})$ and solve System \eqref{trembling hand} to get a map $\sigma(t) = \sigma(t,\bm{x})$, then $\sigma$ is the unique solution to the parabolic equation
	\begin{equation} \label{transport in R^dN ep}
		\begin{aligned}
			\partial_{t} \sigma_{N}(t, \bm{x}) - \frac{\varepsilon^{2N}}{2}\Delta\sigma_{N}(t, \bm{x}) + \sum_{j=1}^{N} D_{x_j} \sigma_{N}(t, \bm{x}) b(t, \sigma_{N}(t, \bm{x}), x_{j}) &= 0,\\
			\sigma_{N}(0, \bm{x}) &= \sigma_{0, N}(\bm{x}),
		\end{aligned}
	\end{equation}
	If Assumptions \ref{sigma0 bounded}, \ref{sigma0 smooth}, \ref{sigma0 cpct supp}, and \ref{b regular} are all in force, then by standard classical theory for parabolic equations, one can show that \eqref{transport in R^dN ep} (which is the same as \eqref{eq:transport disc viscous}) has a unique classical solution $\sigma_{N,\varepsilon}$.
	Indeed, by the maximum principle, $\abs{\sigma_{N,\varepsilon}}$ is a priori bounded by $\max \abs{\sigma_{0}}$.
	By standard bootstrapping arguments, we obtain a priori estimates on $\sigma_{N,\varepsilon}$ first in Sobolev spaces using \cite[Theorem IV.9.1]{ladyzhenskaia1968linear} and then in H\"older spaces using \cite[Theorem IV.5.1]{ladyzhenskaia1968linear}.
	An application of the Leray-Schauder Theorem then gives existence of solutions, and uniqueness follows from the maximum principle.
	We omit the details.
	To connect the PDE system with the $N$-player game, we observe that by \cite{pardoux1999forward}, $\sigma_{N,\varepsilon}(t,\bm{x}) = \bar{\sigma}(-t)$ where $\bar{\sigma}$ is obtained from solving System \eqref{trembling hand}.
	Hence $\sigma_{N,\varepsilon}(t,\bm{x})$ is the desired ``trembling hand equilibrium."
	
	When $\sigma_N$ is indeed a classical solution to \eqref{transport in R^dN}, we might expect that $\sigma_{N,\varepsilon} \to \sigma_N$ uniformly as $\varepsilon \to 0$.
	However, we are interested in the case when Assumption \ref{monotonicity} in particular is not in force, and therefore we can expect neither uniqueness of Nash equilibrium nor continuous solutions of \eqref{transport in R^dN}.
	Instead, we turn to entropy solutions.
	
	Let us introduce the following notation, which allows us to rewrite the transport equations in the form of balance laws.
	\begin{equation} \label{bold notation}
		\begin{split}
			B(t,x,u) &= \int_0^u b(t,x,s)\dif s,\\
			\bm{b}(t,\bm{x},u) &= \del{b(t,x_1,u),\ldots,b(t,x_n,u)},\\
			\bm{B}(t,\bm{x},u) &= \del{B(t,x_1,u),\ldots,B(t,x_n,u)}.
		\end{split}
	\end{equation}
	We observe that \eqref{transport in R^dN} and \eqref{transport in R^dN ep} can be written
	\begin{equation} \label{eq:transport disc}
		\left\{
		\begin{split}
			&\partial_t \sigma_N(t,\bm{x}) + \operatorname{div}_{\bm{x}} \del{\bm{B}(t,\bm{x},\sigma_N(t,\bm{x}))} = \operatorname{div}_{\bm{x}}\bm{B}(t,\bm{x},\sigma_N(t,\bm{x})), &\text{in}~(0,T) \times (\bb{R}^d)^N,\\
			&\sigma_N(0,\bm{x}) = \sigma_{0,N}(\bm{x}), &\text{in}~ (\bb{R}^d)^N.
		\end{split}
		\right.
	\end{equation}
	and
	\begin{equation} \label{eq:transport disc viscous}
		\left\{
		\begin{split}
			&\partial_t \sigma_N(t,\bm{x}) + \operatorname{div}_{\bm{x}} \del{\bm{B}(t,\bm{x},\sigma_N(t,\bm{x}))} = \operatorname{div}_{\bm{x}}\bm{B}(t,\bm{x},\sigma_N(t,\bm{x}))  +  \frac{\varepsilon^{2N}}{2} \Delta_{\bm{x}}\sigma_N(t,\bm{x}), &\text{in}~(0,T) \times (\bb{R}^d)^N,\\
			&\sigma_N(0,\bm{x}) = \sigma_{0,N}(\bm{x}), &\text{in}~ (\bb{R}^d)^N,
		\end{split}
		\right.
	\end{equation}
	respectively.
	
	The Kru\v{z}kov definition of entropy solutions  is as follows \cite{kruvzkov1970first}.
	We say $\sigma_N$ is an entropy solution to \eqref{transport in R^dN} (equivalently \eqref{eq:transport disc}) provided that
	\begin{multline} \label{eq:entropy}
		\int_0^T \int_{(\bb{R}^d)^N} \left(S\del{\sigma_N(t,\bm{x})}\partial_t \phi(t,\bm{x}) + \sum_{j=1}^N D_{x_j}\phi(t,\bm{x})Q(t,x_j,\sigma_N(t,\bm{x}))\right.\\
		\left. + \sum_{j=1}^N H(t,x_j,\sigma_N(t,\bm{x}))\phi(t,\bm{x})\right)\dif \bm{x} \dif t
		+ \int_{\bb{R}^N} S\del{\sigma_{0,N}(\bm{x})}\phi(0,\bm{x})\dif \bm{x} \geq 0 \\
		\forall \phi \in \s{C}^\infty_c\del{\intco{0,T} \times (\bb{R}^d)^N}, \ \phi \geq 0
	\end{multline}
	whenever the functions $S$, $Q$, and $H$ are given by
	\begin{equation} \label{renormalization functions}
		\begin{split}
			S(u) &= \abs{u-k}, \quad Q(t,x,u) = \operatorname{sgn}(u-k)\del{B(t,x,u) - B(t,x,k)}, \\
			H(t,x,u) &= \operatorname{sgn}(u-k)\operatorname{div}_x B(t,x,u)
		\end{split}
	\end{equation}
	for some $k \in \bb{R}$.
	
	Under Assumptions \ref{sigma0 bounded}, \ref{sigma0 smooth}, \ref{sigma0 cpct supp}, and \ref{b regular}, the seminal results of \cite{kruvzkov1970first} show that \eqref{transport in R^dN} (equivalently \eqref{eq:transport disc}) has a unique entropy solution $\sigma_N$.
	Moreover, $\sigma_{N,\varepsilon}$, the solution of \eqref{transport in R^dN ep} (which is the same as \eqref{eq:transport disc viscous}) also satisfies entropy inequalities, namely
	\begin{multline} \label{eq:entropy visc}
		\int_0^T \int_{(\bb{R}^d)^N} \left(S\del{\sigma_{N,\varepsilon}(t,\bm{x})}\partial_t \phi(t,\bm{x}) - \frac{\varepsilon^{2N}}{2} D_{\bm{x}} \sigma_{N,\varepsilon}(t,\bm{x}) \cdot D_{\bm{x}}\phi(t,\bm{x})
		+ \sum_{j=1}^N D_{x_j}\phi(t,\bm{x})Q(t,x_j,\sigma_{N,\varepsilon}(t,\bm{x}))\right.\\
		\left. + \sum_{j=1}^N H(t,x_j,\sigma_{N,\varepsilon}(t,\bm{x}))\phi(t,\bm{x})\right)\dif \bm{x} \dif t
		+ \int_{\bb{R}^N} S\del{\sigma_{0,N}(\bm{x})}\phi(0,\bm{x})\dif \bm{x} \geq 0 \\
		\forall \phi \in \s{C}^\infty_c\del{\intco{0,T} \times (\bb{R}^d)^N}, \ \phi \geq 0.
	\end{multline}
	We are now ready to state our main result.
	
	\begin{theorem}
		\label{thm:error estimates}
		Let $b$ satisfy Assumptions \ref{b regular}, and let $\sigma_0$ be a bounded and continuously differentiable function.	
		Let $\sigma_N$ be the entropy solution of \eqref{transport in R^dN}, and let $\sigma_{N,\varepsilon}$ be the classical solution of \eqref{transport in R^dN ep}.
		Then for every $t \geq 0$,
		\begin{equation}
			\label{main error estimate}
			\int_{(\bb{R}^d)^N} \abs{\sigma_N(t,\bm{x}) - \sigma_{N,\varepsilon}(t,\bm{x})}\dif \bm{x}
			\leq 
			C_b\varepsilon^N (Nd)^{7/2}e^{L_bd(2N+1)t}\del{1 + \int_{(\bb{R}^d)^N} \abs{D_{\bm{x}}\sigma_{0,N}(\bm{x})}\dif \bm{x}},
		\end{equation} 
		where $L_b = \enVert{D_x b}_\infty$, and $C_b$ is a constant depending on $M_b$, $L_b$, and $H_B$ from Assumption \ref{b regular}.
	\end{theorem}
	
	The proof of Theorem \ref{thm:error estimates} is given in Section \ref{thm:error estimates}.
	We note that such error estimates are now well-known going back to Kuznetsov \cite{kuznetsov1976accuracy}, see for instance \cite{bouchut1998kruzkov,cockburn1999continuous,evje2002error} and references therein.
	Our main contribution is to explicitly account for the dependence on $N$, which represents the number of players in the game but in the transport equations appears in the spatial dimension $Nd$.
	Let us introduce a seminorm on measurable functions $f:\sr{P}_2(\bb{R}^d) \to \bb{R}$ given by
	\begin{equation}
		\label{seminorm}
		\enVert{f} := \limsup_{N \to \infty} \del{\int_{(\bb{R}^d)^N} \abs{f(\mu^N_{\bm{x}})}^N\dif \bm{x}}.
	\end{equation}
	The following corollary makes the estimate \eqref{main error estimate} homogeneous with respect to $N$:
	\begin{corollary}
		\label{cor:error estimates}
		Let $\sigma_N$ and $\sigma_{N,\varepsilon}$ be as in Theorem \ref{thm:error estimates}.
		Let $\sigma:\intco{0,\infty} \times \sr{P}_2(\bb{R}^d) \to \bb{R}$ and $\sigma_\varepsilon:\intco{0,\infty} \times \sr{P}_2(\bb{R}^d) \to \bb{R}$ be measurable functions satisfying $\sigma(t,\mu^N_{\bm{x}}) = \sigma_N(t,\bm{x})$ and $\sigma_\varepsilon(t,\mu^N_{\bm{x}}) = \sigma_{N,\varepsilon}(t,\bm{x})$ for all $N \in \bb{N}$ and all $\bm{x} \in (\bb{R}^d)^N$.
		Then for every $t \geq 0$,
		\begin{equation} \label{homogeneous error estimate}
			\enVert{\sigma(t,\cdot) - \sigma_\varepsilon(t,\cdot)}
			\leq \varepsilon\enVert{\sigma_0}_u e^{2L_b d t}\del{1 + R^d\del{2\pi e/d}^{d/2}}
		\end{equation}
		where $\enVert{\sigma_0}_u := \sup_m \abs{\sigma_0(m)}$.
	\end{corollary}
	We prove Corollary \ref{cor:error estimates} in Section \ref{sec:convergence}.
	
	The interpretation of this main result is as follows.
	$\sigma_\varepsilon(t,m)$ gives the solution to the perturbed game \eqref{trembling hand} for any discrete measure $m = \mu^N_{\bm{x}}$ with $N$ and $\bm{x}$ arbitrary. It has a limit as $\varepsilon \to 0$ in sense of Corollary \ref{cor:error estimates}.
	This new function, $\sigma(t,m)$, can be interpreted as a weak solution to the mean field equilibrium condition \eqref{eq:equilibrium}.
	It is in general not clear whether or not, for an arbitrary discrete measure $m = \mu^N_{\bm{x}}$, $\sigma(t,m)$ is in fact a Nash equilibrium; see \cite{graber2024some} for examples showing the answer can sometimes be ``no."
	Nevertheless, in terms of rational response, $\sigma(t,m)$ represents the parameter that players should anticipate if they expect all players to interpret this parameter as a martingale with vanishingly small variance.
	In this sense, it can be interpreted as a ``selection" of a mean field equilibrium, even if it is sometimes a fictitious one.
	
	\begin{remark} \label{rem:conjecture}
		It seems natural to conjecture that, under Assumptions \ref{b regular}, \ref{monotonicity b} and \ref{diagonal} on the vector field $b$, the entropy solutions $\sigma_N(t,\bm{x})$ give Nash equilibria for a.e.~$\bm{x}$, i.e.~\eqref{eq:equilibrium} is satisfied with $m = \mu^N_{\bm{x}}$.
		The intuition behind this conjecture is that Assumptions \ref{monotonicity b} and \ref{diagonal} generalize the condition for one-dimensional conservation laws that the flux function should be strictly convex.
		One then expects that shocks do not meet to form rarefaction waves, which is what would cause \eqref{eq:equilibrium} to be violated.
		This will be a topic for future research.
	\end{remark}

	\section{Proof of Theorem \ref{thm:error estimates}} \label{sec:estimates}
	
	The proof of Theorem \ref{thm:error estimates} largely follows the same outline as in \cite{kruvzkov1970first}.
	The main we point we need to emphasize is how the estimates depend on the dimension of the space, which in turn corresponds to the number of points in the empirical measure.
	Our estimates will thus hinge on a few basic observations from vector calculus.
	Let $F:\bb{R}^d \to \bb{R}^d$ be a vector field, and suppose $\bm{F}:(\bb{R}^d)^N \to (\bb{R}^d)^N$  is given by $\bm{F}(\bm{x}) = (F(x_1),\ldots,F(x_N))$.
	Define
	\begin{equation}
		\enVert{F}_\infty = \operatorname{ess sup}_x |F(x)|,
	\end{equation}
	where as usual $|\cdot|$ denotes the Euclidean norm.
	It follows that
	\begin{equation}
		\enVert{\bm{F}}_\infty = \sqrt{N}\enVert{F}_\infty.
	\end{equation}
	Now consider the derivative $DF(x)$ at a point $x \in \bb{R}^d$.
	We define the norm $\abs{DF(x)}$ by
	\begin{equation}
		\abs{DF(x)} = \sup_{|u| = 1} \abs{DF(x)u}.
	\end{equation}
	We then define
	\begin{equation}
		\enVert{DF}_\infty = \operatorname{ess sup}_x |DF(x)|.
	\end{equation}
	Since $D\bm{F}(\bm{x})$ is diagonal, one sees that
	\begin{equation}
		\enVert{D\bm{F}}_\infty = \enVert{DF}_\infty.
	\end{equation}
	Similarly, we define
	\begin{equation}
		\begin{split}
			\abs{D^2 F(x)} &= \sup_{\abs{u}=1, \abs{v}=1}\abs{D^2F(x)[u,v]},\\
			\enVert{D^2 F}_\infty &= \operatorname{ess sup}_x |D^2F(x)|
		\end{split}
	\end{equation}
	and observe that
	\begin{equation}
		\enVert{D^2 \bm{F}}_\infty = \enVert{D^2 F}_\infty.
	\end{equation}
	On the other hand, if we consider the divergence operator, we see that $\operatorname{div} \bm{F}(\bm{x}) = \sum_{j=1}^N \operatorname{div} F(x_j)$, and thus
	\begin{equation}
		\enVert{\operatorname{div} \bm{F}}_\infty = N\enVert{\operatorname{div} F}_\infty.
	\end{equation}
	It is also worth noting that the optimal estimate for the divergence in terms of the derivative is given by
	\begin{equation}
		\enVert{\operatorname{div} F}_\infty \leq d\enVert{DF}_\infty.
	\end{equation}
	
	Let us now recall the definitions of $\bm{b}$, $B$, and $\bm{B}$ from \eqref{bold notation}.
	Since $\bm{B}$ is the flux function from \eqref{eq:transport disc}, which is \eqref{transport in R^dN} written as a balance law, the proof of Theorem \ref{thm:error estimates} will depend on estimates of $\bm{B}$ and its derivatives, which we collect here: from Assumption \ref{b regular}, we have
	\begin{equation}
		\label{Bestimates}
		\begin{aligned}
			\enVert{\partial_\sigma \bm{B}}_\infty = \enVert{\bm{b}}_\infty &\leq  M_b \sqrt{N},\\
			\enVert{\partial_\sigma D_{\bm{x}}\bm{B}}_\infty = \enVert{D_{\bm{x}}\bm{b}}_\infty &\leq L_b,\\
			\enVert{\partial_\sigma \operatorname{div}_{\bm{x}}\bm{B}}_\infty = \enVert{\operatorname{div}_{\bm{x}}\bm{b}}_\infty &\leq L_b Nd,\\
			\enVert{D_{\bm{xx}}^2\bm{B}}_\infty &\leq H_B,\\
			\enVert{D_{\bm{x}}\operatorname{div}_{\bm{x}}\bm{B}}_\infty &\leq H_BNd
		\end{aligned}
	\end{equation}
	To prove Theorem \ref{thm:error estimates}, our goal is to estimate $\int_{(\bb{R}^d)^N} \abs{\sigma(t,\bm{x}) - \sigma_\varepsilon(t,\bm{x})}\dif \bm{x}$, where $\sigma$ is the entropy solution of \eqref{transport in R^dN} and $\sigma_\varepsilon$ is the classical solution of \eqref{transport in R^dN ep}; we omit the subscript $N$ while deriving these $L^1$ estimates.

	\subsection{$BV$ estimates}
	
	\begin{lemma} \label{lem:parabolic stability}
		Let $\sigma$ be the solution of the parabolic equation \eqref{eq:transport disc viscous}.
		Then 
		\begin{equation} \label{parabolic BV}
			\int_{(\bb{R}^d)^N} \abs{D_{\bm{x}}\sigma(t,\bm{x})}\dif \bm{x}
			\leq \sqrt{Nd}e^{L_b(N+1)d t}\int_{(\bb{R}^d)^N} \abs{D_{\bm{x}}\sigma_{0,N}(\bm{x})}\dif \bm{x}
		\end{equation}
		for all $t \geq 0$.
	\end{lemma}
	
	\begin{proof}
		The proof essentially follows the argument found in \cite[Section 4.1, case B]{kruvzkov1970first}, though here we will pay attention to the particular structure of our problem.
		Let $v_j^k = \partial_{x_j^k} \sigma$.
		Then $v_j^k$ satisfies (in the sense of distributions)
		\begin{equation}
			\partial_t v_j^k - \frac{\varepsilon^{2N}}{2} \Delta_{\bm{x}} v_j^k + \operatorname{div}_{\bm{x}} \del{\bm{b}(t,\bm{x},\sigma)v_j^k}
			+ \operatorname{div}_{\bm{x}} \bm{b}(t,\bm{x},\sigma)v_j^k
			= - \partial_{x_j^k}\bm{b}(t,\bm{x},\sigma) \cdot v
			= - \partial_{x^k}b(t,x_j,\sigma) \cdot v_j
		\end{equation}
		where $v = (v_1,\ldots,v_n)$ and $v_j = (v_j^1,\ldots,v_j^d)$.
		Fix $\bar{\psi}_j^k \in L^\infty((\bb{R}^d)^N)$ and $\tau > 0$, and let $\psi_j^k$ satisfy the dual equation
		\begin{equation}
			-\partial_t \psi_j^k - \frac{\varepsilon^{2N}}{2} \Delta_{\bm{x}} \psi_j^k - D_{\bm{x}}\psi_j^k \cdot \bm{b}(t,\bm{x},\sigma)
			+ \operatorname{div}_{\bm{x}} \bm{b}(t,\bm{x},\sigma)\psi_j^k = 0
		\end{equation}
		with condition $\psi_j^k(\tau,\bm{x}) = \bar{\psi}_j^k(\bm{x})$.
		By the maximum principle, and recalling \eqref{Bestimates},
		\begin{equation}
			\enVert{\psi_j^k(t,\cdot)}_\infty \leq \enVert{\bar{\psi}_j^k}_\infty e^{\enVert{\operatorname{div}_{\bm{x}} \bm{b}}_\infty (\tau-t)}
			\leq \enVert{\bar{\psi}_j^k}_\infty e^{NdL_b (\tau-t)}.
		\end{equation}
		Integrating by parts, we get
		\begin{multline}
			\int_{(\bb{R}^d)^N} v_j^k(\tau,\bm{x})\bar{\psi}_j^k(\bm{x})\dif \bm{x}\\
			= \int_{(\bb{R}^d)^N} \partial_{x_j^k}\sigma_{0,N}(\bm{x})\psi_j^k(0,\bm{x})\dif \bm{x}
			- \int_0^\tau \int_{(\bb{R}^d)^N} \partial_{x^k}b(t,x_j,\sigma(t,\bm{x})) \cdot v_j(t,\bm{x}) \psi_j^k(t,\bm{x})\dif \bm{x}\dif t\\
			\leq \enVert{\bar{\psi}_j^k}_\infty e^{NdL_b \tau}\int_{(\bb{R}^d)^N} \abs{\partial_{x_j^k}\sigma_{0,N}(\bm{x})}\dif \bm{x}
			+ L_b \enVert{\bar{\psi}_j^k}_\infty \int_0^\tau \int_{(\bb{R}^d)^N} e^{NdL_b (\tau-t)}\abs{v_j(t,\bm{x})} \dif \bm{x}\dif t.
		\end{multline}
		Since $\bar{\psi}_j^k$ is arbitrary,
		\begin{equation}
			\int_{(\bb{R}^d)^N} \abs{v_j^k(\tau,\bm{x})}\dif \bm{x}
			\leq  e^{NdL_b \tau}\int_{(\bb{R}^d)^N} \abs{\partial_{x_j^k}\sigma_{0,N}(\bm{x})}\dif \bm{x}
			+ L_b\int_0^\tau \int_{(\bb{R}^d)^N} e^{NdL_b (\tau-t)}\abs{v_j(t,\bm{x})} \dif \bm{x}\dif t.
		\end{equation}
		Now sum over $k = 1,\ldots,d$.
		Using $\abs{v_j} \leq \sum_k \abs{v_j^k} \leq \sqrt{d}\abs{v_j}$, we see that
		\begin{equation}
			e^{-NdL_b \tau}\int_{(\bb{R}^d)^N} \abs{v_j(\tau,\bm{x})}\dif \bm{x}
			\leq  \sqrt{d}\int_{(\bb{R}^d)^N} \abs{D_{x_j}\sigma_{0,N}(\bm{x})}\dif \bm{x}
			+ L_bd\int_0^\tau \int_{(\bb{R}^d)^N} e^{-NdL_b t}\abs{v_j(t,\bm{x})} \dif \bm{x}\dif t.
		\end{equation}
		Since $\tau > 0$ is arbitrary, by Gronwall's Lemma, this implies
		\begin{equation}
			e^{-NdL_b \tau}\int_{(\bb{R}^d)^N} \abs{v_j(\tau,\bm{x})}\dif \bm{x}
			\leq  \sqrt{d}e^{L_bd \tau}\int_{(\bb{R}^d)^N} \abs{D_{x_j}\sigma_{0,N}(\bm{x})}\dif \bm{x}.
		\end{equation}
		Since $|v| = \del{\sum_j |v_j|^2}^{1/2}$ we get
		\begin{equation}
			\int_{(\bb{R}^d)^N} \abs{v(\tau,\bm{x})}\dif \bm{x}
			\leq  \sqrt{Nd}e^{L_b(N+1)d \tau}\int_{(\bb{R}^d)^N} \abs{D_{\bm{x}}\sigma_{0,N}(\bm{x})}\dif \bm{x},
		\end{equation}
		which is precisely \eqref{parabolic BV}.
	\end{proof}

	\begin{lemma} \label{lem:BV}
		Let $\sigma$ be the unique entropy solution of \eqref{eq:transport disc}.
		Then  for all $t \geq 0$,
		\begin{equation} \label{BV space}
			\int_{(\bb{R}^d)^N} \abs{D_{\bm{x}}\sigma(t,\bm{x})}\dif \bm{x}
			\leq \sqrt{Nd}e^{L_b d (N+1)t}\int_{(\bb{R}^d)^N} \abs{D_{\bm{x}}\sigma_{0,N}(\bm{x})}\dif \bm{x}
		\end{equation}
	\end{lemma}
	
	\begin{proof}
		Let $\sigma_\varepsilon$ be the solution of the parabolic equation \eqref{eq:transport disc viscous}.
		By the proof of \cite[Theorem 4]{kruvzkov1970first} (see also \cite[Theorem 6.2.2]{dafermos16}), $\sigma_\varepsilon(t,\cdot)$ is bounded in $BV$, uniformly in $\varepsilon$ and $t$, and $\sigma_\varepsilon \to \sigma$ both pointwise and in $C\del{[0,T];L^1_{loc}((\bb{R}^d)^N)}$.
		Since $\sigma_\varepsilon$ satisfies \eqref{parabolic BV}, we let $\varepsilon \to 0$ in these inequalities to deduce that $\sigma$ satisfies \eqref{BV space}, as desired.
	\end{proof}

	\subsection{$L^1$ stability} \label{sec:L^1 stability}
	
	To get our initial $L^1$ estimates in Theorem \ref{thm:error estimates}, we employ the doubling of variables technique going back to Kru\v kov and Kuznetsov \cite{kruvzkov1970first,kuznetsov1976accuracy}.
	Fix $0 < \tau_0 < \tau < T$ and $r > 0$.
	For $\alpha,\gamma,\delta \in (0,1)$ small enough, we can define a test function $\varphi = \varphi(t,\bm{x},t',\bm{x}')$ by
	\begin{equation} \label{test function}
		\begin{split}
			\varphi(t,\bm{x},t',\bm{x}') &:= \psi_{\alpha}\del{\frac{t+t'}{2}}\rho_{\gamma,\delta}(t-t',\bm{x}-\bm{x}'),\\
			\rho(t,\bm{x}) = \rho_{\gamma,\delta}(t,\bm{x}) &:= \gamma^{-1}\delta^{-Nd}\rho(\gamma^{-1}t)\prod_{k=1}^{N}\prod_{\ell=1}^d\rho(\delta^{-1}x_k^\ell),\\
			\psi(t) = \psi_{\alpha}(t) &:= \zeta_\alpha(t-\tau) - \zeta_\alpha(t-\tau_0),\\
			\zeta_\eta(t) &:= \zeta(\eta^{-1}t) \quad \forall \eta > 0,
		\end{split}
	\end{equation}
	where $\rho$ is some non-negative smooth function with support in $[-1,1]$ such that $\int \rho = 1$, $\int \abs{\rho'} \leq 2$ and $\int \abs{\rho''} \leq 4$, and where $\zeta$ is a smooth function such that $\zeta(t) = 1$ for $t < 0$, $\zeta(t) = 0$ for $t \geq 1$, $\zeta$ is decreasing on $(0,1)$, $\abs{\zeta'} \leq 2$ and $\abs{\zeta''} \leq 4$.
	We use $\varphi$ as a test function in \eqref{eq:entropy}, with respect to variables $(t,\bm{x})$ and with entropy $S(\sigma) = \abs{\sigma - \sigma_\varepsilon(t',\bm{x}')}$; and in \eqref{eq:entropy visc}, with respect to variables $(t',\bm{x}')$ and with entropy $S(\sigma_\varepsilon) = \abs{\sigma_\varepsilon - \sigma(t,\bm{x})}$.
	Integrating the inequalities and then adding them (cf.~\cite[Section 3]{kruvzkov1970first}), we get
	\begin{multline} \label{error estimate1}
		\int_0^T \int_0^T \int_{(\bb{R}^d)^N}\int_{(\bb{R}^d)^N} \biggl( \abs{\sigma(t,\bm{x}) - \sigma_\varepsilon(t',\bm{x}')}\del{\partial_t \varphi + \partial_{t'} \varphi}
		\\
		+
		D_{\bm{x}}\varphi \cdot \operatorname{sgn}\del{\sigma(t,\bm{x}) - \sigma_\varepsilon(t',\bm{x}')}\del{\bm{B}(t,\bm{x},\sigma(t,\bm{x})) - \bm{B}(t,\bm{x},\sigma_\varepsilon(t',\bm{x}'))}
		\\
		+ 
		D_{\bm{x}'}\varphi \cdot \operatorname{sgn}\del{\sigma(t,\bm{x}) - \sigma_\varepsilon(t',\bm{x}')}\del{\bm{B}(t',\bm{x}',\sigma(t,\bm{x})) - \bm{B}(t',\bm{x}',\sigma_\varepsilon(t',\bm{x}'))}
		\\
		+ 
		\operatorname{sgn}\del{\sigma(t,\bm{x}) - \sigma_\varepsilon(t',\bm{x}')}\del{\operatorname{div}_{\bm{x}}\bm{B}(t,\bm{x},\sigma(t,\bm{x})) - \operatorname{div}_{\bm{x}}\bm{B}(t,\bm{x},\sigma_\varepsilon(t',\bm{x}'))}\varphi \\
		+ 
		\operatorname{sgn}\del{\sigma(t,\bm{x}) - \sigma_\varepsilon(t',\bm{x}')}\del{\operatorname{div}_{\bm{x}'}\bm{B}(t',\bm{x}',\sigma(t,\bm{x})) - \operatorname{div}_{\bm{x}'}\bm{B}(t',\bm{x}',\sigma_\varepsilon(t',\bm{x}'))}\varphi \biggr) \dif \bm{x}\dif \bm{x}'\dif t \dif t'
		\\
		\geq \frac{\varepsilon^{2N}}{2} \int_0^T \int_0^T \int_{(\bb{R}^d)^N}\int_{(\bb{R}^d)^N}
		\operatorname{sgn}\del{\sigma(t,\bm{x}) - \sigma_\varepsilon(t',\bm{x}')}D_{\bm{x}'} \sigma_\varepsilon(t',\bm{x}') \cdot D_{\bm{x}'} \varphi \dif \bm{x}\dif \bm{x}'\dif t \dif t'.
	\end{multline}
	Let us estimate the right-hand side of this inequality.
	We have $$D_{\bm{x}'} \phi = -\psi_{\alpha}\del{\frac{t+t'}{2}}D_{\bm{x}'}\rho_{\gamma,\delta}(t-t',\bm{x}-\bm{x}'),$$ $\abs{\psi_\alpha} \leq 1$, and
	\begin{equation} \label{Dxrho estimate}
		\int_0^T \int_{(\bb{R}^d)^N} \abs{D_{\bm{x}} \rho(t,\bm{x})}\dif \bm{x}\dif t \leq \delta^{-1}Nd,
	\end{equation}
	which we derive from the simple estimate
	\begin{equation*}
		\abs{D_{\bm{x}} \rho(t,\bm{x})} \leq \gamma^{-1}\delta^{-Nd-1}\rho(\gamma^{-1}t)\sum_{i=1}^N\sum_{j=1}^d\rho'(\delta^{-1}x_i^j)\prod_{(k,\ell) \neq (i,j)}\rho(\delta^{-1}x_k^\ell).
	\end{equation*}
	Combining this with \eqref{BV space}, Equation \eqref{error estimate1} becomes
	\begin{multline} \label{error estimate2}
		\int_0^T \int_0^T \int_{(\bb{R}^d)^N}\int_{(\bb{R}^d)^N} \biggl( \abs{\sigma(t,\bm{x}) - \sigma_\varepsilon(t',\bm{x}')}\del{\partial_t \varphi + \partial_{t'} \varphi}
		\\
		+
		D_{\bm{x}}\varphi \cdot \operatorname{sgn}\del{\sigma(t,\bm{x}) - \sigma_\varepsilon(t',\bm{x}')}\del{\bm{B}(t,\bm{x},\sigma(t,\bm{x})) - \bm{B}(t,\bm{x},\sigma_\varepsilon(t',\bm{x}'))}
		\\
		+ 
		D_{\bm{x}'}\varphi \cdot \operatorname{sgn}\del{\sigma(t,\bm{x}) - \sigma_\varepsilon(t',\bm{x}')}\del{\bm{B}(t',\bm{x}',\sigma(t,\bm{x})) - \bm{B}(t',\bm{x}',\sigma_\varepsilon(t',\bm{x}'))}
		\\
		+ 
		\operatorname{sgn}\del{\sigma(t,\bm{x}) - \sigma_\varepsilon(t',\bm{x}')}\del{\operatorname{div}_{\bm{x}}\bm{B}(t,\bm{x},\sigma(t,\bm{x})) - \operatorname{div}_{\bm{x}}\bm{B}(t,\bm{x},\sigma_\varepsilon(t',\bm{x}'))}\varphi \\
		+ 
		\operatorname{sgn}\del{\sigma(t,\bm{x}) - \sigma_\varepsilon(t',\bm{x}')}\del{\operatorname{div}_{\bm{x}'}\bm{B}(t',\bm{x}',\sigma(t,\bm{x})) - \operatorname{div}_{\bm{x}'}\bm{B}(t',\bm{x}',\sigma_\varepsilon(t',\bm{x}'))}\varphi \biggr) \dif \bm{x}\dif \bm{x}'\dif t \dif t'
		\\
		\geq  - \frac{\varepsilon^{2N}}{2\delta} N^2d^{3/2}M_b e^{L_bd(N+1)\del{\tau + \alpha + \gamma}}\int_{(\bb{R}^d)^N} \abs{D_{\bm{x}}\sigma_{0,N}(\bm{x})}\dif \bm{x}.
	\end{multline}
	Next, we will ``symmetrize" the left-hand side by subtracting the following quantities:
	\begin{equation}
		\begin{aligned}
			I_1 &:= \int_0^T \int_0^T \int_{(\bb{R}^d)^N}\int_{(\bb{R}^d)^N} \biggl(\operatorname{sgn}\del{\sigma(t,\bm{x}) - \sigma_\varepsilon(t',\bm{x}')}\del{\bm{B}(t',\bm{x}',\sigma_\varepsilon(t',\bm{x}')) - \bm{B}(t,\bm{x},\sigma_\varepsilon(t',\bm{x}'))}\cdot D_{\bm{x}}\varphi \\
			&\quad -\operatorname{sgn}\del{\sigma(t,\bm{x}) - \sigma_\varepsilon(t',\bm{x}')}\operatorname{div}_{\bm{x}}\bm{B}(t,\bm{x},\sigma_\varepsilon(t',\bm{x}'))\varphi \biggr)\dif \bm{x}\dif \bm{x}'\dif t \dif t',\\
			I_2 &:= \int_0^T \int_0^T \int_{(\bb{R}^d)^N}\int_{(\bb{R}^d)^N} \biggl(\operatorname{sgn}\del{\sigma(t,\bm{x}) - \sigma_\varepsilon(t',\bm{x}')}\del{\bm{B}(t',\bm{x}',\sigma(t,\bm{x})) - \bm{B}(t,\bm{x},\sigma(t,\bm{x}))}\cdot D_{\bm{x}'}\varphi \\
			&\quad +\operatorname{sgn}\del{\sigma(t,\bm{x}) - \sigma_\varepsilon(t',\bm{x}')}\operatorname{div}_{\bm{x}'}\bm{B}(t',\bm{x}',\sigma(t,\bm{x}))\varphi \biggr)\dif \bm{x}\dif \bm{x}'\dif t \dif t'.
		\end{aligned}
	\end{equation}
	Equation \eqref{error estimate2} becomes
	\begin{multline} 
		\int_0^T \int_0^T \int_{(\bb{R}^d)^N}\int_{(\bb{R}^d)^N} \biggl( \abs{\sigma(t,\bm{x}) - \sigma_\varepsilon(t',\bm{x}')}\del{\partial_t \varphi + \partial_{t'} \varphi}
		\\
		+
		\del{D_{\bm{x}}\varphi + D_{\bm{x}'}\varphi} \cdot \operatorname{sgn}\del{\sigma(t,\bm{x}) - \sigma_\varepsilon(t',\bm{x}')}\del{\bm{B}(t,\bm{x},\sigma(t,\bm{x})) - \bm{B}(t',\bm{x}',\sigma_\varepsilon(t',\bm{x}'))}
		\\
		+ 
		\operatorname{sgn}\del{\sigma(t,\bm{x}) - \sigma_\varepsilon(t',\bm{x}')}\del{\operatorname{div}_{\bm{x}}\bm{B}(t,\bm{x},\sigma(t,\bm{x}))  - \operatorname{div}_{\bm{x}'}\bm{B}(t',\bm{x}',\sigma_\varepsilon(t',\bm{x}'))}\varphi \biggr) \dif \bm{x}\dif \bm{x}'\dif t \dif t'
		\\
		\geq  -I_1 - I_2 - \frac{\varepsilon^{2N}}{2\delta} N^2d^{3/2}M_b e^{L_bd(N+1)\del{\tau + \alpha + \gamma}}\int_{(\bb{R}^d)^N} \abs{D_{\bm{x}}\sigma_{0,N}(\bm{x})}\dif \bm{x},
	\end{multline}
	which, because $D_{\bm{x}} \phi + D_{\bm{x}'} \phi = 0$, can be more concisely written
	\begin{multline} \label{error estimate3}
		\int_0^T \int_0^T \int_{(\bb{R}^d)^N}\int_{(\bb{R}^d)^N} \biggl( \abs{\sigma(t,\bm{x}) - \sigma_\varepsilon(t',\bm{x}')}\del{\partial_t \varphi + \partial_{t'} \varphi}
		\\
		+ 
		\operatorname{sgn}\del{\sigma(t,\bm{x}) - \sigma_\varepsilon(t',\bm{x}')}\del{\operatorname{div}_{\bm{x}}\bm{B}(t,\bm{x},\sigma(t,\bm{x}))  - \operatorname{div}_{\bm{x}'}\bm{B}(t',\bm{x}',\sigma_\varepsilon(t',\bm{x}'))}\varphi \biggr) \dif \bm{x}\dif \bm{x}'\dif t \dif t'
		\\
		\geq  -I_1 - I_2 - \frac{\varepsilon^{2N}}{2\delta} N^2d^{3/2}M_b e^{L_bd(N+1)\del{\tau + \alpha + \gamma}}\int_{(\bb{R}^d)^N} \abs{D_{\bm{x}}\sigma_{0,N}(\bm{x})}\dif \bm{x}.
	\end{multline}
	Our next step is to estimate $I_1$ and $I_2$.
	We rewrite
	\begin{equation}
		\label{I1I2}
		\begin{split}
			I_1 &= \int_0^T \int_0^T \int_{(\bb{R}^d)^N}\int_{(\bb{R}^d)^N}  \operatorname{sgn}\del{\sigma(t,\bm{x}) - \sigma_\varepsilon(t',\bm{x}')}J_\varepsilon(t,t',\bm{x},\bm{x}')\psi\del{\frac{t+t'}{2}}\dif \bm{x}\dif \bm{x}'\dif t \dif t',\\
			I_2 &= -\int_0^T \int_0^T \int_{(\bb{R}^d)^N}\int_{(\bb{R}^d)^N}  \operatorname{sgn}\del{\sigma(t,\bm{x}) - \sigma_\varepsilon(t',\bm{x}')}J(t,t',\bm{x},\bm{x}')\psi\del{\frac{t+t'}{2}}\dif \bm{x}\dif \bm{x}'\dif t \dif t',
		\end{split}
	\end{equation}
	where
	\begin{equation}		
		\begin{split}
			J_\varepsilon(t,t',\bm{x},\bm{x}') &= \del{\bm{B}(t',\bm{x}',\sigma_\varepsilon(t',\bm{x}')) - \bm{B}(t,\bm{x},\sigma_\varepsilon(t',\bm{x}'))} \cdot D_{\bm{x}}\rho\del{t-t',\bm{x}-\bm{x}'}\\
			& \quad	- 	\operatorname{div}_{\bm{x}}\bm{B}(t,\bm{x},\sigma_\varepsilon(t',\bm{x}'))\rho\del{t-t',\bm{x}-\bm{x}'},\\
			J(t,t',\bm{x},\bm{x}') &= \bm{B}(t',\bm{x}',\sigma(t,\bm{x})) - \bm{B}(t,\bm{x},\sigma(t,\bm{x}))\cdot D_{\bm{x}}\rho\del{t-t',\bm{x}-\bm{x}'}\\
			&\quad 
			- \operatorname{div}_{\bm{x}}\bm{B}(t,\bm{x},\sigma(t,\bm{x}))\rho\del{t-t',\bm{x}-\bm{x}'}.
		\end{split}
	\end{equation}
	We will approximate $J_\varepsilon$ with
	\begin{multline}
		K_\varepsilon(t,t',\bm{x},\bm{x}') =
		D_{\bm{x}}\rho\del{t-t',\bm{x}-\bm{x}'}D_{\bm{x}'}\bm{B}(t',\bm{x}',\sigma_\varepsilon(t',\bm{x}'))(\bm{x}'-\bm{x})
		\\
		- 	\operatorname{div}_{\bm{x}'}\bm{B}(t',\bm{x}',\sigma_\varepsilon(t',\bm{x}'))\rho\del{t-t',\bm{x}-\bm{x}'},
	\end{multline}
	which, after a little vector calculus, becomes
	\begin{equation}
		K_\varepsilon(t,t',\bm{x},\bm{x}') = 
		\operatorname{tr}\sbr{D_{\bm{x}'}\bm{B}(t',\bm{x}',\sigma_\varepsilon(t',\bm{x}'))D_{\bm{x}}\del{(\bm{x}'-\bm{x})\rho\del{t-t',\bm{x}-\bm{x}'}}}.
	\end{equation}
	Recall that $\rho = \rho_{\gamma,\delta}$ is supported on $[-\gamma,\gamma] \times [-\delta,\delta]^{Nd}$, and for $\bm{x} \in [\delta,\delta]^{Nd}$ we have $\abs{\bm{x}} \leq (Nd)^{1/2}\delta$. Using \eqref{Bestimates}, we have
	\begin{multline} \label{Jep - Kep}
		\abs{J_\varepsilon(t,t',\bm{x},\bm{x}') - K_\varepsilon(t,t',\bm{x},\bm{x}')}
		\leq \del{H_BNd\delta^2 + \omega(\gamma)}\abs{D_{\bm{x}}\rho\del{t-t',\bm{x}-\bm{x}'}}
		\\	+ \del{H_B (Nd)^{3/2}\delta + \omega(\gamma)} \rho\del{t-t',\bm{x}-\bm{x}'},
	\end{multline}
	where both here and in what follows $\omega(\gamma)$ is an arbitrary modulus with respect to $\gamma$, i.e.~some non-decreasing function $\omega:\intco{0,\infty} \to \intco{0,\infty}$ such that $\lim_{\gamma \to 0} \omega(\gamma) = 0$; its precise definition can change from line to line.
	(Here we are using the continuity of $B$ with respect to $t$.)
	Define
	\begin{equation}
		\tilde I_1 = \int_0^T \int_0^T \int_{(\bb{R}^d)^N}\int_{(\bb{R}^d)^N}  \operatorname{sgn}\del{\sigma(t,\bm{x}) - \sigma_\varepsilon(t',\bm{x}')}K_\varepsilon(t,t',\bm{x},\bm{x}')\psi(t')\dif \bm{x}\dif \bm{x}'\dif t \dif t'.
	\end{equation}
	We deduce from \eqref{Jep - Kep} and \eqref{Dxrho estimate} that
	\begin{equation} \label{I1-tI1}
		\abs{I_1 - \tilde I_1} \leq 2H_B (Nd)^2 \delta + \omega(\gamma).
	\end{equation}
	Similarly, we define
	\begin{equation}
		\begin{aligned}
			K(t,t',\bm{x},\bm{x}') &= \operatorname{tr}\sbr{D_{\bm{x}'}\bm{B}(t',\bm{x}',\sigma(t,\bm{x}))D_{\bm{x}}\del{(\bm{x}'-\bm{x})\rho\del{t-t',\bm{x}-\bm{x}'}}},\\
			\tilde I_2 &= \int_0^T \int_0^T \int_{(\bb{R}^d)^N}\int_{(\bb{R}^d)^N}  \operatorname{sgn}\del{\sigma(t,\bm{x}) - \sigma_\varepsilon(t',\bm{x}')}K(t,t',\bm{x},\bm{x}')\psi(t')\dif \bm{x}\dif \bm{x}'\dif t \dif t'.
		\end{aligned}
	\end{equation}
	By the same reasoning as above, we get
	\begin{equation} \label{I2-tI2}
		\abs{I_2 - \tilde I_2} \leq 2H_B (Nd)^2 \delta + \omega(\gamma).
	\end{equation}
	Now we write
	\begin{multline}
		\tilde I_1 + \tilde I_2 \\
		= \iiiint \operatorname{tr}\sbr{G(t',\bm{x}',\sigma(t,\bm{x}),\sigma_\varepsilon(t',\bm{x}')) D_{\bm{x}}\del{(\bm{x}'-\bm{x})\rho\del{t-t',\bm{x}-\bm{x}'}}}\psi(t')\dif \bm{x}\dif \bm{x}'\dif t \dif t'
	\end{multline}
	where
	\begin{equation}
		G(t',\bm{x}',\sigma,\sigma') := \operatorname{sgn}(\sigma-\sigma') \del{D_{\bm{x}'} \bm{B}(t',\bm{x}',\sigma) - D_{\bm{x}'} \bm{B}(t',\bm{x}',\sigma')}.
	\end{equation}
	Observe that
	\begin{equation}
		\iiiint \operatorname{tr}\sbr{G(t',\bm{x}',\sigma(t,\bm{x}'),\sigma_\varepsilon(t',\bm{x}')) D_{\bm{x}}\del{(\bm{x}'-\bm{x})\rho\del{t-t',\bm{x}-\bm{x}'}}}\psi(t')\dif \bm{x}\dif \bm{x}'\dif t \dif t' = 0.
	\end{equation}
	On the other hand, by \eqref{Bestimates} we have
	\begin{equation}
		\abs{G(t',\bm{x}',\sigma(t,\bm{x}),\sigma_\varepsilon(t',\bm{x}')) - G(t',\bm{x}',\sigma(t,\bm{x}'),\sigma_\varepsilon(t',\bm{x}'))}
		\leq L_b\abs{\sigma(t,\bm{x}) - \sigma(t,\bm{x}')}.
	\end{equation}
	Recalling that for bounded linear operators $A$ and $B$ on $\bb{R}^n$ we have $\abs{\operatorname{tr}(AB)} \leq n|A||B|$, we have
	\begin{equation}
		\begin{aligned}
			\abs{\tilde I_1 + \tilde I_2}
			&\leq L_bNd\iiiint \abs{\sigma(t,\bm{x}) - \sigma(t,\bm{x}')}\abs{D_{\bm{x}}\del{(\bm{x}'-\bm{x})\rho\del{t-t',\bm{x}-\bm{x}'}}}\dif \bm{x}\dif \bm{x}'\dif t \dif t'\\
			&\leq \iiiint \int_0^1 \abs{D_{\bm{x}}\sigma(t,\lambda \bm{x} + (1-\lambda)\bm{x}')}\dif \lambda \biggl(\abs{\bm{x} - \bm{x}'}^2 \abs{D_{\bm{x}} \rho(t-t',\bm{x} - \bm{x}')}\\
			& \phantom{\iiiint \int} \quad + \abs{\bm{x} - \bm{x}'}\rho(t-t',\bm{x} - \bm{x}')\biggr)\dif \bm{x}\dif \bm{x}'\dif t \dif t'
		\end{aligned}
	\end{equation}
	After a change of variable and using \eqref{BV space} and \eqref{Dxrho estimate}, we have
	\begin{equation} \label{tI1 + tI2}
		\abs{\tilde I_1 + \tilde I_2}
		\leq 2L_b(Nd)^{7/2} \delta e^{L_b d (N+1)\del{\tau + \alpha + \gamma}}\int_{(\bb{R}^d)^N} \abs{D_{\bm{x}}\sigma_{0,N}(\bm{x})}\dif \bm{x}.
	\end{equation}
	Combining \eqref{I1-tI1}, \eqref{I2-tI2}, and \eqref{tI1 + tI2} into \eqref{error estimate3}, we get
	\begin{multline} \label{error estimate4}
		\int_0^T \int_0^T \int_{(\bb{R}^d)^N}\int_{(\bb{R}^d)^N} \biggl( \abs{\sigma(t,\bm{x}) - \sigma_\varepsilon(t',\bm{x}')}\del{\partial_t \varphi + \partial_{t'} \varphi}
		\\
		+ 
		\operatorname{sgn}\del{\sigma(t,\bm{x}) - \sigma_\varepsilon(t',\bm{x}')}\del{\operatorname{div}_{\bm{x}}\bm{B}(t,\bm{x},\sigma(t,\bm{x}))  - \operatorname{div}_{\bm{x}'}\bm{B}(t',\bm{x}',\sigma_\varepsilon(t',\bm{x}'))}\varphi \biggr) \dif \bm{x}\dif \bm{x}'\dif t \dif t'
		\\
		\geq -4H_B (Nd)^2 \delta - \omega(\gamma) -\del{2\delta L_b(Nd)^{7/2} + \frac{\varepsilon^{2N}}{2\delta} M_bN^2d^{3/2}} e^{L_b d (N+1)\del{\tau + \alpha + \gamma}}\int_{(\bb{R}^d)^N} \abs{D_{\bm{x}}\sigma_{0,N}(\bm{x})}\dif \bm{x}.
	\end{multline}
	Now observe that
	\begin{equation}
		\begin{aligned}
			\partial_t \varphi(t,\bm{x},t',\bm{x}') + \partial_{t'} \varphi(t,\bm{x},t',\bm{x}') &= \partial_t\psi_{\alpha}\del{\frac{t+t'}{2}}\rho_{\gamma,\delta}(t-t',\bm{x}-\bm{x}')\\
			&= \del{\zeta_\alpha'\del{\frac{t+t'}{2} - \tau} - \zeta_\alpha'\del{\frac{t+t'}{2} - \tau_0}}\rho_{\gamma,\delta}(t-t',\bm{x}-\bm{x}')
		\end{aligned}
	\end{equation}
	We also have, using \eqref{Bestimates},
	\begin{multline}
		\iiiint \operatorname{sgn}\del{\sigma(t,\bm{x}) - \sigma_\varepsilon(t',\bm{x}')}\del{\operatorname{div}_{\bm{x}}\bm{B}(t,\bm{x},\sigma(t,\bm{x}))  - \operatorname{div}_{\bm{x}'}\bm{B}(t',\bm{x}',\sigma_\varepsilon(t',\bm{x}'))}\varphi \dif \bm{x}\dif \bm{x}'\dif t \dif t'\\
		\leq L_bNd\iiiint \abs{\sigma(t,\bm{x}) - \sigma_\varepsilon(t',\bm{x}')}\psi_{\alpha}\del{\frac{t+t'}{2}}\rho_{\gamma,\delta}(t-t',\bm{x}-\bm{x}')\dif \bm{x}\dif \bm{x}'\dif t \dif t'\\
		+ H_B(Nd)^{3/2}\delta + \omega(\gamma). 
	\end{multline}
	Using these estimates, we let $\gamma \to 0$ and then $\alpha \to 0$ in \eqref{error estimate4}; using the fact that $\sigma$ and $\sigma_\varepsilon$ are in $C([0,T];L^1(\bb{R}^{Nd}))$ and appealing to standard convolution approximation arguments, we get
	\begin{multline} \label{error estimate5}
		\int_{(\bb{R}^d)^N}\int_{(\bb{R}^d)^N} \del{\abs{\sigma(\tau_0,\bm{x}) - \sigma_\varepsilon(\tau_0,\bm{x}')}
			- \abs{\sigma(\tau,\bm{x}) - \sigma_\varepsilon(\tau,\bm{x}')}}\rho_{\delta}(\bm{x}-\bm{x}')\dif \bm{x}\dif \bm{x}'
		\\
		+ 
		L_bNd\int_0^\tau \int_{(\bb{R}^d)^N}\int_{(\bb{R}^d)^N}\abs{\sigma(t,\bm{x}) - \sigma_\varepsilon(t,\bm{x}')}\rho_{\delta}(\bm{x}-\bm{x}')\dif \bm{x}\dif \bm{x}'\dif t
		\\
		\geq -5H_B (Nd)^2 \delta  -\del{2\delta L_b(Nd)^{7/2} + \frac{\varepsilon^{2N}}{2\delta} M_bN^2d^{3/2}} e^{L_b d (N+1)\tau}\int_{(\bb{R}^d)^N} \abs{D_{\bm{x}}\sigma_{0,N}(\bm{x})}\dif \bm{x}
	\end{multline}
	where $\rho_{\delta}(\bm{x}) := \delta^{-nd}\prod_{k=1}^{n}\prod_{\ell=1}^d\rho(\delta^{-1}x_k^\ell)$.
	Using \eqref{BV space}, we have
	\begin{multline}
		\abs{\int_{(\bb{R}^d)^N}\int_{(\bb{R}^d)^N}\abs{\sigma(t,\bm{x}) - \sigma_\varepsilon(t,\bm{x}')}\rho_{\delta}(\bm{x}-\bm{x}')\dif \bm{x}\dif \bm{x}' - \int_{(\bb{R}^d)^N}\abs{\sigma(t,\bm{x}) - \sigma_\varepsilon(t,\bm{x})}\dif \bm{x}}\\
		\leq 
		Nd \delta e^{L_b d (N+1)t}\int_{(\bb{R}^d)^N} \abs{D_{\bm{x}}\sigma_{0,N}(\bm{x})}\dif \bm{x}
	\end{multline}
	for every $t$.
	We use this estimate in \eqref{error estimate5} and let $\tau_0 \to 0$ to get
	\begin{multline}
		\int_{(\bb{R}^d)^N}\abs{\sigma(\tau,\bm{x}) - \sigma_\varepsilon(\tau,\bm{x})}\dif \bm{x}
		\leq L_bNd\int_0^\tau \int_{(\bb{R}^d)^N}\abs{\sigma(t,\bm{x}) - \sigma_\varepsilon(t,\bm{x})}\dif \bm{x}\dif t\\
		+ 5H_B (Nd)^2 \delta  +\del{2\delta L_b(Nd)^{7/2} + \frac{\varepsilon^{2N}}{2\delta} M_bN^2d^{3/2} + 3Nd\delta} e^{L_b d (N+1)\tau}\int_{(\bb{R}^d)^N} \abs{D_{\bm{x}}\sigma_{0,N}(\bm{x})}\dif \bm{x}.
	\end{multline}
	Since $\tau$ is arbitrary, Gronwall's Lemma implies
	\begin{multline} \label{error estimate6}
		\int_{(\bb{R}^d)^N}\abs{\sigma(\tau,\bm{x}) - \sigma_\varepsilon(\tau,\bm{x})}\dif \bm{x}
		\leq 5H_B (Nd)^2 \delta e^{L_b Nd \tau}\\
		+   \del{2\delta L_b(Nd)^{7/2} + \frac{\varepsilon^{2N}}{2\delta} M_bN^2d^{3/2} + 3Nd\delta} e^{L_b d (2N+1)\tau}\int_{(\bb{R}^d)^N} \abs{D_{\bm{x}}\sigma_{0,N}(\bm{x})}\dif \bm{x}.
	\end{multline}
	Take $\delta = \frac{\varepsilon^N}{2}$ in \eqref{error estimate6}.
	By simple estimates on the constants, we arrive at
	\begin{multline} \label{error estimate6}
		\int_{(\bb{R}^d)^N}\abs{\sigma(\tau,\bm{x}) - \sigma_\varepsilon(\tau,\bm{x})}\dif \bm{x}\\
		\leq 
		\varepsilon^N (Nd)^{7/2}e^{L_b d (2N+1)\tau}\del{2H_B + \del{L_b + M_b + 2} \int_{(\bb{R}^d)^N} \abs{D_{\bm{x}}\sigma_{0,N}(\bm{x})}\dif \bm{x}},
	\end{multline}
	thus proving Theorem \ref{thm:error estimates}.

	\subsection{Homogenizing the error estimates with respect to $N$} \label{sec:convergence}
	
	Our goal in this section is to prove Corollary \ref{cor:error estimates}, which is interpreted as an error estimate that is uniform with respect to $N$, i.e.~with respect to all possible empirical measures.
	We make use of the following seminorm.
	For $f:\sr{P}_2(\bb{R}^d) \to \bb{R}$ measurable, we define
	\begin{equation} \label{norm}
		\enVert{f} := \limsup_{N \to \infty} \del{\int_{(\bb{R}^d)^N} \abs{f_N(\bm{x})}^N\dif \bm{x}}^{1/N}
		= \limsup_{N \to \infty} \del{\int_{(\bb{R}^d)^N} \abs{f(\mu^N_{\bm{x}})}^N\dif \bm{x}}^{1/N}.
	\end{equation}
	It is straightforward to check that \eqref{norm} defines a seminorm on the space of measurable functions.
	
	Before proceeding to the proof of Corollary \ref{cor:error estimates}, it will be useful to point out some properties of the seminorm defined in \eqref{norm}.
	First, let us consider $f = 1_{B_R(\delta_0)}$, where $\delta_0$ denotes the Dirac mass at the origin, so that
	\begin{equation}
		B_R(\delta_0) = \cbr{\mu \in \sr{P}_2(\bb{R}^d) : \del{\int \abs{x}^2 \dif \mu(x)}^{1/2} \leq R}.
	\end{equation}
	Thus
	\begin{equation}
		f(\mu) = 1_{B_R(\delta_0)}(\mu) = \begin{cases}
			1 &\text{if}~\del{\int \abs{x}^2 \dif \mu(x)}^{1/2} \leq R,\\
			0 &\text{if}~\del{\int \abs{x}^2 \dif \mu(x)}^{1/2} > R.
		\end{cases}
	\end{equation}
	It is instructive to compute $\enVert{1_{B_R(\delta_0)}}$.
	Observe that
	\begin{equation}
		f_N(\bm{x}) = 1_{B_R(\delta_0)}(\mu^N_{\bm{x}}) = \begin{cases}
			1 &\text{if}~\abs{\bm{x}} \leq \sqrt{N}R,\\
			0 &\text{if}~\abs{\bm{x}} > \sqrt{N} R
		\end{cases}
		= 1_{B_{\sqrt{N}R}(0)}(\bm{x}).
	\end{equation}
	Thus $\enVert{1_{B_R(\delta_0)}}$ essentially comes down to the following result.
	\begin{lemma} \label{lem:ball in n dim}
		Let $\s{L}^n$ denote $n$-dimensional Lebesgue measure and $B_1$ a ball of radius 1.
		Then
		\begin{equation}
			\lim_{n \to \infty} n^{1/2}\s{L}^n(B_1)^{1/n} = \sqrt{2\pi e}.
		\end{equation}
	\end{lemma}
	
	\begin{proof}
		Define $A_n := n^{1/2}\s{L}^n(B_1)^{1/n}$
		Recall the well-known formula
		\begin{equation}
			\s{L}^n(B_1) = \frac{\pi^{n/2}}{\Gamma\del{\frac{n}{2}+1}}.
		\end{equation}
		Thus $A_n = (n\pi)^{1/2}\Gamma\del{\frac{n}{2}+1}^{-1/n}$.
		
		By a change of variables, we have
		\begin{equation}
			\Gamma\del{\frac{n}{2}+1} = \int_0^\infty t^{n/2}e^{-t}\dif t = n^{\frac{n}{2}+1}\int_0^\infty \del{t^{1/2}e^{-t}}^n\dif t.
		\end{equation}
		The function $t^{1/2}e^{-t}$ is bounded on $\intco{0,\infty}$, and a simple calculation shows that its maximum is $(2e)^{-1/2}$.
		We deduce that
		\begin{equation}
			n^{-1/2}\Gamma\del{\frac{n}{2}+1}^{1/n} = n^{1/n}\del{\int_0^\infty \del{t^{1/2}e^{-t}}^n\dif t}^{1/n} \to (2e)^{-1/2},
		\end{equation}
		from which the claim follows.
	\end{proof}
	\begin{corollary}
		\label{cor:norm BR}
		$\enVert{1_{B_R(\delta_0)}} = R^d(2\pi e/d)^{d/2}$.
		Thus, if $f:\sr{P}_2(\bb{R}^d) \to \bb{R}$ is any bounded measurable function whose support is contained in the set $B_R(\delta_0)$, then $\enVert{f} \leq R^d(2\pi e/d)^{d/2}\enVert{f}_u$, where $\enVert{f}_u = \sup_m \abs{f(m)}$.
	\end{corollary}
	
	\begin{proof}
		Let $f = 1_{B_R(\delta_0)}$.
		From Lemma \ref{lem:ball in n dim} we have
		\begin{equation}
			\begin{aligned}
				\del{\int_{(\bb{R}^{d})^N} \abs{f_N(\bm{x})}^N \dif \bm{x}}^{1/Nd}
				& = \s{L}^{Nd}(B_{N^{1/2}R}(0))^{1/Nd}\\
				&= N^{1/2}R\s{L}^{Nd}(B_{1}(0))^{1/Nd}\\
				& \to R(2\pi e/d)^{1/2},
			\end{aligned}
		\end{equation}
		and the claim follows by taking both sides to the $d$ power.
		The claim for general bounded $f$ with support in $B_R(\delta_0)$ follows easily.
	\end{proof}
	It is worth noticing the following result, which follows immediately from Corollary \ref{cor:norm BR}.
	\begin{corollary} \label{cor:uniform}
		Let $\{f_n\}$ be a sequences of bounded measurable functions on $\sr{P}_2(\bb{R}^d)$ with support in $B_R(\delta_0)$ that converge uniformly to $f$.
		Then $\enVert{f_n - f} \to 0$.
	\end{corollary}
	The converse of Corollary \ref{cor:uniform} does not hold, since if $f_n = 1_{B_{1/n}(\delta_0)}$ then $\enVert{f_n} \to 0$ but $f_n$ does not converge to 0 uniformly.
	Thus convergence in the seminorm is strictly weaker than uniform convergence for sequences of functions with bounded support.
	
	Let $\sigma_N$ denote the entropy solution of \eqref{transport in R^dN} and $\sigma_{N,\varepsilon}$ the classical solution of \eqref{transport in R^dN ep}.
	By the maximum principle, $\abs{\sigma_N}$ and $\abs{\sigma_{N,\varepsilon}}$ are both bounded by $\enVert{\sigma_0}_u := \sup_m \abs{\sigma_0(m)}$.
	It follows from Theorem \ref{thm:error estimates} that
	\begin{multline} \label{sigmaN - sigmaNep to N}
		\del{\int_{(\bb{R}^d)^N} \abs{\sigma_N(t,\bm{x}) - \sigma_{N,\varepsilon}(t,\bm{x})}^N \dif \bm{x}}^{1/N}\\
		\leq C_b^{1/N}\varepsilon\enVert{\sigma_0}_u^{1-1/N} (Nd)^{7/2N}e^{L_b(2+1/N) d t}\del{1 + \del{\int_{(\bb{R}^d)^N} \abs{D_{\bm{x}} \sigma_{0,N}(\bm{x})}\dif \bm{x}}^{1/N}}.
	\end{multline}
	Notice also that
	\begin{equation}
		\abs{D_{\bm{x}}\sigma_{0,N}(\bm{x})}
		= \del{\sum_{j=1}^N \abs{\frac{1}{N}D_m \sigma_0(\mu^N_{\bm{x}},x_j)}^2}^{1/2}
		\leq N^{-1/2}\enVert{D_m \sigma_0}_u,
	\end{equation}
	and so, because of Assumption \ref{sigma0 cpct supp}, we have
	\begin{equation}
		\int_{(\bb{R}^d)^N} \abs{D_{\bm{x}}\sigma_{0,N}(\bm{x})}\dif \bm{x}
		\leq \int_{B_{\sqrt{N}R}} N^{-1/2}\enVert{D_m \sigma_0}_u \dif \bm{x}
		= N^{(Nd - 1)/2}R^{Nd} \s{L}^{Nd}(B_1)\enVert{D_m \sigma_0}_u.
	\end{equation}
	By Lemma \ref{lem:ball in n dim}, we have
	\begin{equation}
		\del{\int_{(\bb{R}^d)^N} \abs{D_{\bm{x}}\sigma_{0,N}(\bm{x})}\dif \bm{x}}^{1/Nd}
		\leq N^{1/2}R \s{L}^{Nd}(B_1)^{1/Nd}\enVert{D_m \sigma_0}_u^{1/Nd} \to R(2\pi e/d)^{1/2}.
	\end{equation}
	Thus, letting $N \to \infty$ in \eqref{sigmaN - sigmaNep to N}, we derive the estimate in Corollary \ref{cor:error estimates}.

	\section{Conclusion} \label{sec:conclusion}
	In this paper we have used a simple class of mean field games to explore the issue of non-uniqueness for Nash equilibrium.
	We have constructed explicit examples where uniqueness can fail, thus demonstrating the need for a selection principle.
	The selection principle we have proposed relies on the classical theory of entropy solution, and it selects an equilibrium whenever the given initial measure is discrete, i.e.~an empirical measure.
	To this we add the caveat that the selection is only defined for ``almost every'' empirical measure, i.e.~for $m = \mu_{\bm{x}}^N$ for a.e.~$\bm{x} \in (\bb{R}^d)^N$, for every $N$.
	This selection can be interpreted as a ``vanishing noise'' selection because of Theorem \ref{thm:error estimates}, and Corollary \ref{cor:error estimates} shows that the error between the selected equilibrium and the perturbed equilibrium from \eqref{trembling hand} can be made uniform with respect to $N$.
	
	We are left, however, with some open questions.
	Our method does not directly provide a selection principle for arbitrary initial measures.
	Since any measure in the Wasserstein space can be approximated by empirical measures, one might expect that a selection principle for empirical measures to extend more generally.
	However, our error estimates rely on $L^1$ estimates in $(\bb{R}^d)^N$, which do not extend naturally to the whole Wasserstein space.
	Thus, a major open question is whether and how a vanishing noise limit could be established for measures that are not discrete.
	A second open question would be to prove the conjecture mentioned in Remark \ref{rem:conjecture}, concerning whether entropy solutions necessarily give genuine equilibria. 
	Finally, concerning the $N$-player game, we know that System \eqref{eq:N player equilibrium} is associated with a system of PDEs.
	While the $N$-player equilibria are expected to converge to mean field equilibria, it is not clear whether one can prove a similar result for entropy solutions to the PDE system.

	To conclude, it is worth remarking that we have seen rather deep mathematical questions arise from a seemingly very simple class of mean field games.
	The questions left open are quite difficult in full generality and will perhaps be best approached using new families of special cases.

	\appendix
	
	\section{Optimal discretization of a measure on $\bb{R}$} \label{sec:optimal discretization}
	
	Fix $m \in \sr{P}_2(\bb{R})$ and $n \in \bb{N}$.
	For any ${\bf x} = (x_1,\ldots,x_n)$ with $x_1 \leq \cdots \leq x_n$ be given and set $\mu_{\bm{x}}^N = \frac{1}{N}\sum_{j=1}^N \delta_{x_j}$.
	Our goal is to find the discrete measure $\mu_{\bm{x}}^N$ that minimizes $W_2(\mu_{\bm{x}}^N,m)$.
	We call this the ``optimal discretization with $n$ points" of $m$.
	
	\subsection{The partition of $m$} \label{sec:partition}
	We will partition the mass of $m$ into $n$ equal portions, from left to right. To be precise, suppose $I_1 = \intoc{-\infty,a_1}, I_n = \intco{a_{N-1},\infty},$ and $I_j = \intcc{a_{j-1},a_j}$ for $j = 2,\ldots,N-1$ for some $a_1 \leq \cdots \leq a_{N-1}$. Note that some of these ``intervals'' may in fact be singleton sets.
	We will show how to choose $a_1,\ldots,a_{N-1}$ and construct measures $m_{N,j}$ on each $I_j$ such that
	\begin{itemize}
		\item $m_{N,j}|_{I_j^\circ} = m|_{I_j^\circ}$ for all $j=1,\ldots,N$, where $I_j^\circ$ denotes the interior of $I_j$;
		\item $m_{N,j}(I_j) = \frac{1}{N}$; and
		\item $m = \sum_{j=1}^N m_{N,j}$.
	\end{itemize} 
	
	We define the sequence $a_1 \leq a_2 \leq \cdots \leq a_{N-1}$ as follows:
	\begin{equation*}
		a_{j} = \min\cbr{x : m\del{\intoc{-\infty,x}} \geq \frac{j}{N}}.
	\end{equation*}
	Now to construct the measures $m_{N,j}$, we proceed as follows.
	\begin{enumerate}
		\item Of necessity, $m_{N,1} = m|_{(-\infty,a_1)} + r_1\delta_{a_1}$ where $r_1 = \frac{1}{N} - m\del{(-\infty,a_1)}$ and, similarly, $m_{N,N} = m|_{(a_{N-1},\infty)} + \ell_n\delta_{a_{N-1}}$ where $\ell_n = \frac{1}{N} - m\del{(a_{N-1},\infty)}$.
		\item Partition the sequence $a_1,\ldots,a_{N-1}$ into $a_{j_1},\ldots,a_{j_2-1},a_{j_2},\ldots,a_{j_3-1},\ldots,a_{j_{k}},\ldots,a_{N-1}$ where for each $i = 1,\ldots,k$, $a_{j_i-1} < a_{j_{i}} = \ldots = a_{j_{i+1}-1}$ (with $j_1 := 1$ and $j_{k+1} := n$); hence $k$ is the number of \emph{distinct} points in the sequence $a_1,\ldots,a_{N-1}$.
		\item For each $j$ such that $a_{j-1} = a_{j}$, we have $I_j = \{a_j\}$, and so of necessity we take $m_{N,j} = \frac{1}{N}\delta_{a_j}$.
		Hence for each $i=1,\ldots,k$, for each $j = j_{i}+1,\ldots,j_{i+1}-1$, we have $m_{N,j} = \frac{1}{N}\delta_{a_{j_i}}$.
		\item For $i = 2,\ldots,k$, let $j = j_i$. Since $m_{N,j_i}$ is defined on $I_{j_i} = [a_{j_i-1},a_{j_i}] = [a_{j_{i-1}},a_{j_i}]$ and $m_{N,j_i}|_{(a_{j_{i-1}},a_{j_i})} = m|_{(a_{j_{i-1}},a_{j_i})}$, of necessity we have $m_{N,j_i} = m|_{(a_{j_{i-1}},a_{j_i})} + \ell_i \delta_{a_{j_{i-1}}} + r_i \delta_{a_{j_i}}$ for some non-Negative numbers $\ell_i$ and $r_i$.
		These must satisfy the following:
		\begin{equation*}
			\begin{cases}
				m\del{(a_{j_{i-1}},a_{j_i})} + \ell_i + r_i = \frac{1}{N}, &\text{to satisfy}~m_{N,j_i}(I_{j_i}) = \frac{1}{N},\\
				r_{i-1} + \frac{j_i - j_{i-1}-1}{N} + \ell_i = m(\{a_{j_{i-1}}\}), &\text{to satisfy}~m(\{a_{j_{i-1}}\}) = \sum_{j=1}^N m_{N,j}(\{a_{j_{i-1}}\}).
			\end{cases}
		\end{equation*}
		It is straightforward to define these constants inductively, namely by defining $\ell_i$ using the second equation and then $r_i$ using the first, and also to see that they will be non-Negative constants that account for all the mass of $m$ concentrated at the points $a_{j_{i}}$.
		This completes the construction.
	\end{enumerate}
	
	As a simple illustration, consider the following graph:
	
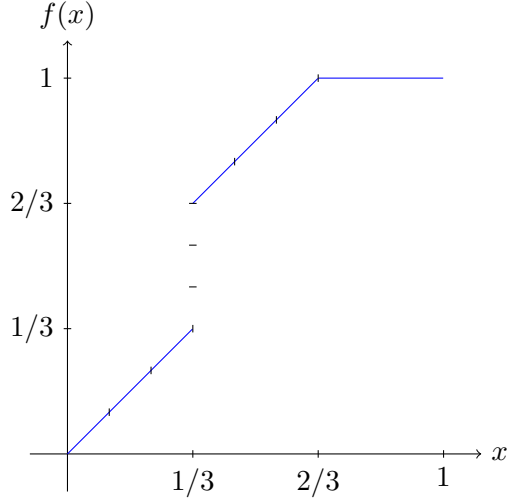
\begin{center}
		\begin{tikzpicture}[scale=5]
			
			% Axis
			\draw[->] (-0.1,0) -- (1.1,0) node[right] {$x$};
			\draw[->] (0,-0.1) -- (0,1.1) node[above] {$f(x)$};
			\draw (1/3,0.01) -- (1/3, -0.01) node[below] {$1/3$};
			\draw (0.01,1/3) -- (-0.01,1/3) node[left] {$1/3$};			
			\draw (2/3,0.01) -- (2/3, -0.01) node[below] {$2/3$};
			\draw (0.01,2/3) -- (-0.01,2/3) node[left] {$2/3$};
			\draw (1,0.01) -- (1, -0.01) node[below] {$1$};
			\draw (0.01,1) -- (-0.01,1) node[left] {$1$};
			
			% Function
			\draw[domain=0:1/3, smooth, samples=100, blue] plot (\x, \x);
			\draw[domain=1/3:2/3, smooth, samples=100, blue] plot (\x, \x + 1/3);
			\draw[domain=2/3:1, smooth, samples=100, blue] plot (\x, 1);
			%\draw[-, blue] (1/3,1/3) -- (1/3,2/3);
			
			% Vertical tick marks
			\foreach \x in {1/9, 2/9, 1/3}
			\draw (\x,\x + 0.01) -- (\x,\x -0.01);		
			\foreach \x in {4/9, 5/9, 2/3}
			\draw (\x,\x + 1/3 + 0.01) -- (\x,\x + 1/3 -0.01);
			
			% Horizontal tick marks
			\foreach \y in {4/9, 5/9, 2/3}
			\draw (1/3 - .01,\y) -- (1/3  + .01,\y);
			
		\end{tikzpicture}
	\end{center}
	Suppose this is the cumulative distribution function of a measure $m$.
	The jump at $x = 1/3$ represents a Dirac mass at $m$.
	To divide $m$ into 9 equal parts from left to right, we will get six densities and three Dirac masses:
	\begin{itemize}
		\item $m_{9,1},m_{9,2},m_{9,3}$ all have density $f(x) = x$ restricted to the intervals $[0,1/9], [1/9,2/9],$ and $[2/9,1/3]$, respectively;
		\item $m_{9,7},m_{9,8},m_{9,9}$ all have density $f(x) = x + 1/3$ restricted to the intervals $[1/3,4/9], [4/9,5/9],$ and $[5/9,2/3]$, respectively;
		\item $m_{9,4},m_{9,5},m_{9,6}$ are equal to the Dirac mass $\frac{1}{9}\delta_{1/3}$.
	\end{itemize}
	
	It is straightforward to see that our construction is unique.
	That is, if $I_1,\ldots,I_n$ are any intervals ordered from left to right on the real line with $m_{N,j}$ satisfying the three properties above, then of necessity both $I_j$ and $m_{N,j}$ are exactly as defined above.
	
	\subsection{Optimal discretization of $m$}
	
	Let $\mu_{\bm{x}}^N = \frac{1}{N}\sum_{j=1}^N \delta_{x_j}$ be any discrete measure with $n$ points. Our goal is to determine how to choose $x_1,\ldots,x_n$ in such a way as to minimize $W_2(\mu_{\bm{x}}^N,m)$. 
	\begin{proposition} \label{prop:transport plan}
		If $\pi \in \Pi(\mu_{\bm{x}}^N,m)$ is an optimal transport plan, then
		\begin{equation} \label{eq:optimal coupling with emp}
			\pi = \sum_{j=1}^N \delta_{x_j} \times m_{N,j}.
		\end{equation}
	\end{proposition}
	
	\begin{proof}
		We begin by observing that $m_j(E) = \pi(\{x_j\} \times E)$ defines a measure on $\bb{R}$ whose total mass is $m_j(\bb{R}) = \mu_{\bm{x}}^N(\{x_j\}) = \frac{1}{N}$.
		Since $\pi$ is supported in $\{x_1,\ldots,x_n\} \times \bb{R}$, it follows that
		\begin{equation}
			\pi = \sum_{j=1}^N \delta_{x_j} \times m_j.
		\end{equation}
		Our goal is to prove that $m_j = m_{N,j}$.
		First, define $E_j$ to be the support of $m_j$.
		Let us show that the sets $E_j$ must be ordered from left to right, i.e.~if $y_j \in E_j$ and $y_{j+1} \in E_{j+1}$ for some $j$ such that $x_j \neq x_{j+1}$ (and hence $x_j < x_{j+1}$), then $y_j \leq y_{j+1}$.
		Cyclical monotonicity of optimal transport plans tells us that
		\begin{equation*}
			\abs{x_j - y_j}^2 + \abs{x_{j+1} - y_{j+1}}^2 \leq \abs{x_j - y_{j+1}}^2 + \abs{x_{j+1} - y_j}^2.
		\end{equation*}
		Rearranging, we deduce
		\begin{equation*}
			(y_{j+1} - y_j)(x_{j+1} - x_j) \geq 0,
		\end{equation*}
		and so $y_{j+1} \geq y_j$, as desired.
		
		Now since $E_1,\ldots,E_n$ are ordered from left to right, the sets $E_1,\ldots,E_{N-1}$ all have upper bounds (and $E_n$ has a lower bound).
		Let $b_j = \sup E_j$ for $j = 1,\ldots,N-1$, and set $I_1 = \intoc{-\infty,b_1}$, $I_n = \intco{b_{N-1},\infty}$, and $I_j = [b_{j-1},b_j]$ for $j = 2,\ldots,N-1$.
		Let us show that $m_j|_{I_j^\circ} = m|_{I_j^\circ}$ for all $j$.
		If $E \subset I_j^\circ$, it follows that $E \cap E_k = \empty$ for all $k \neq j$, so $m(E) = \pi(\bb{R} \times E) = \sum_k \delta_{x_k} \times m_k(\bb{R} \times E) = \sum_k m_k(E) = m_j(E)$.
		It follows that the intervals $I_j$ and measures $m_j$ satisfy the properties of the partition from Section \ref{sec:partition}, so $m_j = m_{N,j}$, as desired.
	\end{proof}
	
	It follows from Proposition \ref{prop:transport plan} that
	\begin{equation}
		W_2(\mu_{\bm{x}}^N,m)^2 = \sum_{j=1}^N \int_{I_j} \abs{x_j - y}^2 \dif m_{N,j}(y),
	\end{equation}
	which is a strictly convex function of $\bm{x}$.
	Hence it has a unique minimizer, which we derive from the first-order condition
	\begin{equation} \label{optimal xj}
		x_j = n\int_{I_j} y \dif m_{N,j}(y).
	\end{equation}
	This is simply the barycenter of the measure $m_{N,j}$. A useful corollary of this fact is that $\mu_{\bm{x}}^N$ has the same barycenter as $m$, i.e.
	\begin{equation} \label{same barycenter}
		\frac{1}{N}\sum_{j=1}^N x_j
		= \sum_{j=1}^N \int y \dif m_{N,j}(y) = \int y \dif m(y).
	\end{equation}
	Finally, let us show that $W_2(\mu_{\bm{x}}^N,m) \to 0$ as $n \to \infty$, where for each $n$ the measure $\mu_{\bm{x}}^N$ is the optimal discretization with $n$ points.
	Let $\varepsilon > 0$.
	There exists $M > 0$ such that
	\begin{equation}
		\int_{[-M,M]^c} x^2 \dif m(x) \leq \varepsilon
		\quad \text{and} \quad
		m([-M,M]) \geq \frac{1}{2}.
	\end{equation}
	We can assume $n \geq 2$.
	Recall the definitions $I_1 = \intoc{-\infty,a_1}, I_n = \intco{a_{N-1},\infty},$ and $I_j = \intcc{a_{j-1},a_j}$ for $j = 2,\ldots,N-1$ from Section \ref{sec:partition}.
	From \eqref{optimal xj} we deduce that $a_{j-1} \leq x_j \leq a_j$ for $j = 2,\ldots,N-1$, and $x_1 \leq a_1$ and $x_n \geq a_n$.
	Let $k$ be the smallest index such that $a_k \geq -M$ and $\ell$ the greatest index such that $a_\ell \leq M$.
	We have
	\begin{equation} \label{discretization estimate 1}
		\sum_{j = k+1}^\ell \int_{I_j} \abs{x - x_j}^2 \dif m_{N,j}(x) \leq \sum_{j = k+1}^\ell \frac{1}{N}\abs{a_j - a_{j-1}}^2
		\leq \frac{1}{N} \del{\sum_{j = k+1}^\ell (a_j - a_{j-1})}^2
		= \frac{1}{N}(a_\ell - a_k)^2.
	\end{equation}
	If $j \leq k-1$ then $a_j < - M < 0$ and so
	\begin{equation} \label{discretization estimate 2}
		\sum_{j=1}^{k-1} \int_{I_j} \abs{x - x_j}^2 \dif m_{N,j}(x)
		\leq \sum_{j=1}^{k-1} \int_{I_j} \abs{x}^2 \dif m_{N,j}(x)
		\leq \int_{(-\infty,-M)} x^2 \dif m(x).
	\end{equation}
	Here we have used the construction of $m_{N,j}$ and the fact that $a_{k-1} < -M$ to deduce that $\sum_{j=1}^{k-1} m_{N,j} \leq m|_{(-\infty,-M)}$.
	Similarly, if $j \geq \ell + 2$ then $a_{j-1} > M > 0$ and so
	\begin{equation} \label{discretization estimate 3}
		\sum_{j=\ell+2}^N \int_{I_j} \abs{x - x_j}^2 \dif m_{N,j}(x)
		\leq \sum_{j=\ell+2}^N \int_{I_j} \abs{x}^2 \dif m_{N,j}(x)
		\leq \int_{(M,\infty)} x^2 \dif m(x).
	\end{equation}
	As for $j = k$, note that since $a_{k-1} < -M$ we necessarily have $a_k \leq M$, since
	\begin{equation}
		m(\intoc{-\infty,M}) \geq m(\intoc{-\infty,a_{k-1}}) + m([-M,M])
		\geq \frac{k-1}{N} + \frac{1}{2} \geq \frac{k}{N}.
	\end{equation}
	Therefore
	\begin{equation} \label{discretization estimate 4}
		\begin{split}
			\int_{I_k} \abs{x - x_k}^2 \dif m_{N,k}(x) 
			&\leq \int_{[a_{k-1},-M)} \abs{x-M}^2 \dif m_{N,k}(x)
			+ \int_{[-M,a_k]} \abs{x-x_k}^2 \dif m_{N,k}(x)\\
			&\leq \int_{(-\infty,-M)} 4x^2 \dif m(x)
			+ \frac{1}{N}(a_k + M)^2.
		\end{split}
	\end{equation}
	In an analogous way, we deduce
	\begin{equation} \label{discretization estimate 5}
		\int_{I_{\ell+1}} \abs{x - x_{\ell+1}}^2 \dif m_{N,\ell+1}(x)
		\leq \int_{(M,\infty)} 4x^2 \dif m(x)
		+ \frac{1}{N}(M - a_\ell)^2.
	\end{equation}
	Adding together \eqref{discretization estimate 1}, \eqref{discretization estimate 2}, \eqref{discretization estimate 3}, \eqref{discretization estimate 4}, and \eqref{discretization estimate 5}, we get
	\begin{equation}
		\sum_{j=1}^N \int_{I_j} \abs{x_j - x}^2 \dif m_{N,j}(x)
		\leq 5\int_{[-M,M]^c} x^2 \dif m(x)
		+ \frac{1}{N}(2M)^2
		\leq 5\varepsilon + \frac{1}{N}(2M)^2.
	\end{equation}
	Letting $n \to \infty$ and then $\varepsilon \to 0$, we obtain the desired convergence.
	\bibliographystyle{alpha}
	\bibliography{../../mybib/mybib}

\newcommand{\etalchar}[1]{$^{#1}$}
\begin{thebibliography}{ACD{\etalchar{+}}21}

\bibitem[ACD{\etalchar{+}}21]{cardaliaguet2021cetraro}
Yves Achdou, Pierre Cardaliaguet, Fran{\c{c}}ois Delarue, Alessio Porretta, and
  Filippo Santambrogio.
\newblock {\em Mean Field Games: Cetraro, Italy 2019}, volume 2281.
\newblock Springer Nature, 2021.

\bibitem[BF19]{bardi2019non}
Martino Bardi and Markus Fischer.
\newblock On non-uniqueness and uniqueness of solutions in finite-horizon mean
  field games.
\newblock {\em ESAIM: Control, Optimisation and Calculus of Variations}, 25:44,
  2019.

\bibitem[BFY13]{bensoussan2013mean}
Alain Bensoussan, Jens Frehse, and Phillip Yam.
\newblock {\em Mean field games and mean field type control theory}.
\newblock Springer, 2013.

\bibitem[BP98]{bouchut1998kruzkov}
Fran{\c{c}}ois Bouchut and Beno{\i}t Perthame.
\newblock Kruzkov’s estimates for scalar conservation laws revisited.
\newblock {\em Transactions of the American Mathematical Society},
  350(7):2847--2870, 1998.

\bibitem[BZ20]{bayraktar2020non}
Erhan Bayraktar and Xin Zhang.
\newblock On non-uniqueness in mean field games.
\newblock {\em Proceedings of the American Mathematical Society},
  148(9):4091--4106, 2020.

\bibitem[Car10]{cardaliaguet2010notes}
Pierre Cardaliaguet.
\newblock Notes on mean field games.
\newblock {\em from P.-L. Lions' lectures at College de France}, 2010.

\bibitem[CD17a]{carmona2017probabilistic}
R~Carmona and F~Delarue.
\newblock {Probabilistic Theory of Mean Field Games: vol. I, Mean Field FBSDEs,
  Control, and Games}.
\newblock {\em Stochastic Analysis and Applications, Springer Verlag}, 2017.

\bibitem[CD17b]{carmona2017probabilisticII}
R~Carmona and F~Delarue.
\newblock {Probabilistic Theory of Mean Field Games: vol. II, Mean Field Games
  with Common Noise and Master Equations}.
\newblock {\em Stochastic Analysis and Applications. Springer Verlag}, 2017.

\bibitem[CD22a]{cecchin2022selection}
Alekos Cecchin and Fran{\c{c}}ois Delarue.
\newblock Selection by vanishing common noise for potential finite state mean
  field games.
\newblock {\em Communications in Partial Differential Equations},
  47(1):89--168, 2022.

\bibitem[CD22b]{cecchin2022weak}
Alekos Cecchin and Fran{\c{c}}ois Delarue.
\newblock Weak solutions to the master equation of potential mean field games.
\newblock {\em arXiv preprint arXiv:2204.04315}, 2022.

\bibitem[CDLL19]{cardaliaguet2019master}
P.~Cardaliaguet, F.~Delarue, J.M. Lasry, and P.L. Lions.
\newblock {\em The Master Equation and the Convergence Problem in Mean Field
  Games: (AMS-201)}.
\newblock Annals of Mathematics Studies. Princeton University Press, 2019.

\bibitem[CF18]{campi2018nplayer}
Luciano Campi and Markus Fischer.
\newblock N-player games and mean-field games with absorption.
\newblock {\em The Annals of Applied Probability}, 28(4):2188--2242, August
  2018.

\bibitem[CG99]{cockburn1999continuous}
B~Cockburn and G~Gripenberg.
\newblock Continuous dependence on the nonlinearities of solutions of
  degenerate parabolic equations.
\newblock {\em Journal of differential equations}, 151(2):231--251, 1999.

\bibitem[Cir16]{cirant2016stationary}
Marco Cirant.
\newblock Stationary focusing mean-field games.
\newblock {\em Communications in Partial Differential Equations},
  41(8):1324--1346, 2016.

\bibitem[CL83]{crandall1983viscosity}
Michael~G Crandall and Pierre-Louis Lions.
\newblock Viscosity solutions of {Hamilton-Jacobi} equations.
\newblock {\em Transactions of the American Mathematical Society},
  277(1):1--42, 1983.

\bibitem[CP19]{cecchin2019convergence}
Alekos Cecchin and Guglielmo Pelino.
\newblock Convergence, fluctuations and large deviations for finite state mean
  field games via the master equation.
\newblock {\em Stochastic Processes and their Applications},
  129(11):4510--4555, 2019.

\bibitem[CR20]{cardaliaguet2020multiple}
Pierre Cardaliaguet and Catherine Rainer.
\newblock An example of multiple mean field limits in ergodic differential
  games.
\newblock {\em Nonlinear Differential Equations and Applications NoDEA}, 27:25,
  2020.

\bibitem[CT19]{cirant2019time}
Marco Cirant and Daniela Tonon.
\newblock Time-dependent focusing mean-field games: the sub-critical case.
\newblock {\em Journal of Dynamics and Differential Equations}, 31:49--79,
  2019.

\bibitem[Daf16]{dafermos16}
C.M. Dafermos.
\newblock {\em Hyperbolic conservation laws in continuum physics}, volume 325
  of {\em Grundlehren der mathematischen Wissenschaften [Fundamental Principles
  of Mathematical Sciences]}.
\newblock Springer-Verlag, Berlin, fourth edition, 2016.

\bibitem[Del19]{delarue2019restoring}
Francois Delarue.
\newblock Restoring uniqueness to mean-field games by randomizing the
  equilibria.
\newblock {\em Stochastics and Partial Differential Equations: Analysis and
  Computations}, 7(4):598--678, 2019.

\bibitem[DT20]{delarue2020selection}
Fran{\c{c}}ois Delarue and Rinel~Foguen Tchuendom.
\newblock Selection of equilibria in a linear quadratic mean-field game.
\newblock {\em Stochastic Processes and their Applications}, 130(2):1000--1040,
  2020.

\bibitem[EK02]{evje2002error}
Steinar Evje and Kenneth~H Karlsen.
\newblock An error estimate for viscous approximate solutions of degenerate
  parabolic equations.
\newblock {\em Journal of Nonlinear Mathematical Physics}, 9(3):262--281, 2002.

\bibitem[G{\etalchar{+}}14]{gomes2014mean}
Diogo~A Gomes et~al.
\newblock Mean field games models--a brief survey.
\newblock {\em Dynamic Games and Applications}, 4(2):110--154, 2014.

\bibitem[GM23]{graber2023monotonicity}
P.~Jameson Graber and Alp\'ar~R. M\'esz\'aros.
\newblock On monotonicity conditions for mean field games.
\newblock {\em Journal of Functional Analysis}, 285(9):110095, 2023.

\bibitem[GM24]{graber2024some}
P.~Jameson Graber and Alp{\'a}r~R. M{\'e}sz{\'a}ros.
\newblock On some mean field games and master equations through the lens of
  conservation laws.
\newblock {\em Mathematische Annalen}, Apr 2024.

\bibitem[Gra25]{graber2025remarks}
P.~Jameson Graber.
\newblock Remarks on potential mean field games.
\newblock {\em Research in the Mathematical Sciences}, 12(1):13, February 2025.

\bibitem[GT19]{gangbo2019differentiability}
Wilfrid Gangbo and Adrian Tudorascu.
\newblock {On differentiability in the Wasserstein space and well-posedness for
  Hamilton–Jacobi equations}.
\newblock {\em Journal de Math\'ematiques Pures et Appliqu\'ees}, 125:119--174,
  2019.

\bibitem[HCM07]{huang2007large}
Minyi Huang, Peter~E Caines, and Roland~P Malham{\'e}.
\newblock Large-population cost-coupled {LQG} problems with nonuniform agents:
  Individual-mass behavior and decentralized $\varepsilon$-nash equilibria.
\newblock {\em Automatic Control, IEEE Transactions on}, 52(9):1560--1571,
  2007.

\bibitem[HMC06]{huang2006large}
Minyi Huang, Roland~P Malham{\'e}, and Peter~E Caines.
\newblock Large population stochastic dynamic games: closed-loop
  {McKean-Vlasov} systems and the nash certainty equivalence principle.
\newblock {\em Communications in Information \& Systems}, 6(3):221--252, 2006.

\bibitem[HS88]{harsanyi1988general}
John~C Harsanyi and Reinhard Selten.
\newblock A general theory of equilibrium selection in games.
\newblock {\em MIT Press Books}, 1, 1988.

\bibitem[Kru67]{kruvzkov1967generalized}
SN~Kru{\v{z}}kov.
\newblock Generalized solutions of nonlinear first order equations with several
  independent variables. ii.
\newblock {\em Mathematics of the USSR-Sbornik}, 1(1):93, 1967.

\bibitem[Kru70]{kruvzkov1970first}
Stanislav~N Kru{\v{z}}kov.
\newblock First order quasilinear equations in several independent variables.
\newblock {\em Mathematics of the USSR-Sbornik}, 10(2):217, 1970.

\bibitem[Kuz76]{kuznetsov1976accuracy}
NN~Kuznetsov.
\newblock Accuracy of some approximate methods for computing the weak solutions
  of a first-order quasi-linear equation.
\newblock {\em USSR Computational Mathematics and Mathematical Physics},
  16(6):105--119, 1976.

\bibitem[Lac16]{lacker2016general}
Daniel Lacker.
\newblock A general characterization of the mean field limit for stochastic
  differential games.
\newblock {\em Probability Theory and Related Fields}, 165(3):581--648, 2016.

\bibitem[Lac20]{lacker2020convergence}
Daniel Lacker.
\newblock On the convergence of closed-loop nash equilibria to the mean field
  game limit.
\newblock {\em The Annals of Applied Probability}, 30(4):1693--1761, 2020.

\bibitem[LL07]{lasry07}
Jean-Michel Lasry and Pierre-Louis Lions.
\newblock Mean field games.
\newblock {\em Japanese Journal of Mathematics}, 2(1):229--260, 2007.

\bibitem[LS23]{lions2023linear}
Pierre-Louis Lions and Benjamin Seeger.
\newblock Linear and nonlinear transport equations with coordinate-wise
  increasing velocity fields.
\newblock {\em arXiv preprint arXiv:2307.05819}, 2023.

\bibitem[LSU68]{ladyzhenskaia1968linear}
Ol'ga~Aleksandrovna Lady{\v z}henskaja, Vsevolod~Alekseevich Solonnikov, and
  Nina~N. Ural'ceva.
\newblock {\em Linear and Quasi-linear Equations of Parabolic Type}, volume~23
  of {\em Translations of mathematical monographs}.
\newblock American Mathematical Society, 1968.

\bibitem[NMT20]{nutz2020convergence}
Marcel Nutz, Jaime~San Martin, and Xiaowei Tan.
\newblock {Convergence to the mean field game limit: A case study}.
\newblock {\em The Annals of Applied Probability}, 30(1):259 -- 286, 2020.

\bibitem[PT99]{pardoux1999forward}
Etienne Pardoux and Shanjian Tang.
\newblock Forward-backward stochastic differential equations and quasilinear
  parabolic pdes.
\newblock {\em Probability Theory and Related Fields}, 114:123--150, 1999.

\bibitem[Sel75]{selten1975}
R.~Selten.
\newblock Reexamination of the perfectness concept for equilibrium points in
  extensive games.
\newblock {\em International Journal of Game Theory}, 4(1):25--55, Mar 1975.

\bibitem[Vil08]{villani2008optimal}
C{\'e}dric Villani.
\newblock {\em Optimal transport: old and new}, volume 338.
\newblock Springer Science \& Business Media, 2008.

\end{thebibliography}
	%\bibliography{mybib}
\end{document}